\documentclass[a4paper]{amsart}
\usepackage[leqno]{amsmath}
\usepackage{amssymb}
\usepackage{amscd}
\usepackage{amsthm}
\usepackage{mathrsfs}
\usepackage{mathtools}
\usepackage{stmaryrd}
\SetSymbolFont{stmry}{bold}{U}{stmry}{m}{n}
\usepackage{bbm}
\usepackage{bm}
\usepackage{color}
\usepackage{accents}
\usepackage{enumerate}
\usepackage{cite}
\usepackage[utf8]{inputenc}
\usepackage[all,cmtip]{xy}
\usepackage{etoolbox}
\usepackage{tikz-cd}
\usepackage{extarrows}

\numberwithin{equation}{section}

\newcommand{\Z}{\ensuremath{\mathbb{Z}}}
\newcommand{\Q}{\ensuremath{\mathbb{Q}}}
\newcommand{\R}{\ensuremath{\mathbb{R}}}
\newcommand{\C}{\ensuremath{\mathbb{C}}}
\newcommand{\A}{\ensuremath{\mathbb{A}}}
\newcommand{\T}{\ensuremath{\mathbb{T}}}

\DeclareMathOperator{\Hom}{Hom}
\DeclareMathOperator{\End}{End}
\DeclareMathOperator{\Ext}{Ext}

\DeclareMathOperator{\Gal}{Gal}
\DeclareMathOperator{\Ker}{ker}

\DeclareMathOperator{\GL}{GL}
\DeclareMathOperator{\PGL}{PGL}

\newcommand{\m}{\ensuremath{\mathfrak{m}}}
\newcommand{\p}{\ensuremath{\mathfrak{p}}}
\newcommand{\q}{\ensuremath{\mathfrak{q}}}

\newcommand{\LI}{\mathcal{L}}

\DeclareMathOperator{\an}{an}

\DeclareMathOperator{\cont}{cont}
\DeclareMathOperator{\ord}{ord}

\DeclareMathOperator{\St}{St}
\DeclareMathOperator{\II}{\mathbb{I}}
\DeclareMathOperator{\cind}{c-ind}

\DeclareMathOperator{\HH}{\ensuremath{H}}
\renewcommand{\det}{\operatorname{det}}
\newcommand{\cf}{{\mathbbm 1}}			
\newcommand{\pinfty}{\ensuremath{^{\p,\infty}}}
\newcommand{\sinfty}{\ensuremath{^{S,\infty}}}
\DeclareMathOperator{\ram}{ram}

\DeclareMathOperator{\disc}{disc}

\newcommand{\into}{\hookrightarrow}
\newcommand{\onto}{\twoheadrightarrow}
\newcommand{\too}{\longrightarrow}								
\newcommand{\mapstoo}{\longmapsto}
\newcommand{\intoo}{\lhook\joinrel\longrightarrow}
\DeclareRobustCommand\ontoo{\relbar\joinrel\twoheadrightarrow}

\newtheorem{Lem}{Lemma}[section]
\makeatletter
\newlength{\@thlabel@width}%
\newcommand{\thmenumhspace}{\settowidth{\@thlabel@width}{\itshape1.}\sbox{\@labels}{\unhbox\@labels\hspace{\dimexpr-\leftmargin+\labelsep+\@thlabel@width-\itemindent}}}
\makeatother
\newtheorem{Pro}[Lem]{Proposition}
\newtheorem{Thm}[Lem]{Theorem}

\newtheorem{Cor}[Lem]{Corollary}
\theoremstyle{definition}

\newtheorem{Def}[Lem]{Definition}
\newtheorem{Rem}[Lem]{Remark}

\author[L. Gehrmann]{Lennart Gehrmann}
\email{lennart.gehrmann@uni-due.de}
\author[M. Pati]{Maria Rosaria Pati}
\email{pati@dima.unige.it}

\title[L-invariants for PGL(2)]{L-invariants for cohomological representations of PGL(2) over arbitrary number fields}
\begin{document}

\begin{abstract}
Let $\pi$ be a cuspidal, cohomological automorphic representation of an inner form $G$ of $\PGL_2$ over a number field $F$ of arbitrary signature.
Further, let $\p$ be a prime of $F$ such that $G$ is split at $\p$ and the local component $\pi_\p$ of $\pi$ at $\p$ is the Steinberg representation.
Assuming that the representation is non-critical at $\p$ we construct automorphic $\LI$-invariants for the representation $\pi$.
If the number field $F$ is totally real, we show that these automorphic $\LI$-invariants agree with the Fontaine--Mazur $\LI$-invariant of the associated $p$-adic Galois representation.
This generalizes a recent result of Spie\ss~respectively Rosso and the first named author from the case of parallel weight $2$ to arbitrary cohomological weights.
\end{abstract}

\maketitle

\tableofcontents

\section*{Introduction}
The purpose of this article is threefold:
first, let $\pi$ be a cuspidal, cohomological automorphic representation of an inner form $G$ of $\PGL_2$ over a number field $F$ of arbitrary signature and $\p$ be a prime of $F$ such that $G$ is split at $\p$ and the local component $\pi_\p$ of $\pi$ at $\p$ is the Steinberg representation.
We want to give a general construction of automorphic $\LI$-invariants (also known as Teitelbaum, Darmon or Orton $\LI$-invariants) for $\pi$.
For representations which are cohomological with respect to the trivial coefficient system, or in other words for forms of parallel weight $2$, these $\LI$-invariants have been defined in general (see for example \cite{Ge2}) but for higher weights they have been defined only in certain situations:
\begin{enumerate}[1.]
\item In case $F=\Q$ and $G$ is compact at infinity by Teitelbaum in \cite{Teitelbaum},
\item in case $F=\Q$ and $G$ is split by Orton in \cite{Orton},
\item in case $F=\Q$ and $G$ is split at infinity by Rotger and Seveso in \cite{RotSev},
\item in case $F$ is imaginary quadratic and $G$ is split by Barrera-Salazar and Williams in \cite{BWi} and
\item in case $F$ is totally real and $G$ is compact at infinity by Chida, Mok and Park in \cite{CMP}.
\end{enumerate}
An obstacle that might have prevented the construction in general is the following:
whereas in case of the trivial coefficient system the representation is always ordinary and therefore non-critical at $\p$, this is no longer true for higher weights.
But note that for the cases 1.-4. above the representation has still non-critical slope at $\p$ and, thus, is non-critical at $\p$ (see the end of Section \ref{overconvergentsec} for a detailed discussion on non-critical slopes).
It seems to the authors of this paper that in \cite{CMP} it is implicitly assumed that the representation is non-critical at $\p$ (see Remark \ref{CMPremark} for more details).
Our first main result (see Definition \ref{defdef}) is the construction of automorphic $\LI$-invariants under the assumption that the representation $\pi$ is non-critical at $\p$.
We point out that our construction of $\LI$-invariants is novel as it does not involve the Bruhat-Tits tree at any stage.

Our second goal is to bridge the gap between works using overconvergent cohomology à la Ash-Stevens, for example \cite{BWi} and \cite{BDJ}, and Spie\ss' more representation-theoretic approach (cf.~\cite{Sp}).
In particular, we show that the non-criticality condition for classes in overconvergent cohomology that is discussed in \cite{BDJ} respectively \cite{BH} is equivalent to a more representation-theoretic one:
assume for the moment that $\pi_\p$ is not necessarily Steinberg but merely has an Iwahori-fixed vector. 
we explain that choosing a $\p$-stabilization of $\pi$, i.e., an Iwahori-fixed vector of $\pi_\p$, yields a cohomology class of a $\p$-arithmetic subgroup of $G(F)$ with values in the dual of a locally algebraic principal series representation of $G(F_\p)$.
Non-criticality is then equivalent to the fact that this class lifts uniquely to a class in the cohomology with values in the continuous dual of the corresponding locally analytic principal series representation (see Proposition \ref{equivalence}).
The main tool to prove this equivalence is the resolution of locally analytic principal series representations by Kohlhaase and Schraen (see \cite{KoSch}).

Finally, we show that, if the number field $F$ is totally real the automorphic $\LI$-invariants attached to $\pi$ agree with the derivatives of the $U_\p$-eigenvalue of a $p$-adic family passing through $\pi$ (cf.~Theorem \ref{mainthm}).
This equality is known in case $F=\Q$ by the work of Bertolini-Darmon-Iovita (see \cite{BDI}) and Seveso (see \cite{Sev}).
For Hilbert modular forms of parallel weight $2$ the equality was recently proven independently by Rosso and the first named author (see \cite{GR}).
As we do not work with general reductive groups as in \textit{loc.cit.~}the arguments simplify substantially making them more accessible to people who are only interested in Hilbert modular forms.
Furthermore, it is known that the derivatives of the $U_\p$-eigenvalue agree with the Fontaine--Mazur $\LI$-invariant of the associated Galois representation, if that Galois representation is non-critical. 
Thus, we deduce the equality of automorphic and Fontaine--Mazur $\LI$-invariants (see \cite{Sp3} for an independent proof of this equality in case of parallel weight $2$). 
The equality of $\LI$-invariants in parallel weight $2$ is necessary for the construction of plectic Stark-Heegner points in recent work of Fornea and the first named author (cf.~\cite{FG}).
This article should be seen as a precursor for defining plectic Stark-Heegner cycles for arbitrary cohomological weights.

\subsection*{Notations}
All rings will be commutative and unital.
The units of a ring $R$ will be denoted by $R^\times.$
If $R$ is a ring and $G$ is a group, we denote the group algebra of $G$ over $R$ by $R[G]$.
The trivial character of any group will be denoted by $\cf$.
Given two sets $X$ and $Y$ we will write $\mathcal{F}(X,Y)$ for the set of all maps from $X$ to $Y$.
If $X$ and $Y$ are topological spaces, we write $C(X,Y)\subseteq \mathcal{F}(X,Y)$ for the set of all continuous maps.
If $Y$ is a topological group, we denote by $C_c(X,Y)\subseteq C(X,Y)$ the subset of functions with compact support.

\subsection*{Acknowledgements}
We would like to thank Yiwen Ding for a enlightening correspondence about non-criticality of Galois representations.
We thank Daniel Barrera Salazar for helpful remarks on an earlier draft of this manuscript.

\subsection*{Setup}
\textit{A number field.}
We fix an algebraic number field $F\subseteq \C$ with ring of integers $\mathcal{O}_F$.
We write $S_\infty$ for the set of infinite places of $F$ and $\Sigma$ for the set of all embeddings from $F$ into the algebraic closure $\overline{\Q}$ of $\Q$ in $\C$.
%We decompose $\Sigma=\Sigma_\R \cup \Sigma_\C$, where $\Sigma_\R$ (resp.~$\Sigma_\C$) denotes the subset of real (resp.~complex) embeddings of $F$. 
The action of complex conjugation on $\Sigma$ will be denoted by $c$.
We write $\delta$ for the number of complex places of $F$.

For any place $\q$ of $F$ we will denote by $F_\q$ the completion of $F$ at $\mathfrak{q}$.
If $\mathfrak{q}$ is a finite place, we let $\mathcal{O}_{\mathfrak{q}}$ denote the valuation ring of $F_{\q}$ and $\ord_{\mathfrak{q}}$ the additive valuation such that $\ord_{\mathfrak{q}}(\varpi)=1$ for any local uniformizer $\varpi\in\mathcal{O}_{\mathfrak{q}}$.
We write $\mathcal{N}(\q)$ for the cardinality of the residue field $\mathcal{O}/\q$.
We normalize the $\q$-adic absolute valuation $|\cdot|_{\q}$ by $|\varpi|_{\p}=\mathcal{N}(\mathfrak{q})^{-1}.$

For a finite set $S$ of places of $F$ we define the ``$S$-truncated adeles'' $\A^{S}$ as the restricted product of all completions $F_{v}$ with $v\notin S$.
In case $S$ is the empty set we drop the superscript $S$ from the notation.
We will often write $\A\sinfty$ instead of $\A^{S\cup S_\infty}$.

If $H$ is an algebraic group over $F$, we will put $H_\q=H(F_\q)$ for any place $\q$ of $F$.
If $S$ is a finite set of places of $F$, we will write $H_S=\prod_{\q\in S} H_\q$.
Further, we abbreviate $H_\infty=H_{S_\infty}$.

\textit{A quaternion algebra.}
We fix a quaternion algebra $D$ over $F$.
We denote by $\ram(D)$ the set of places of $F$ at which $D$ is ramified and put
 \[
 \disc(D)=\prod_{\q\in \ram(D), \q\nmid \infty} \q.
 \]
Let $D^\times$ be the group of units of $D$ considered as an algebraic group over $F$.
The centre $Z\subseteq D^\times$ is naturally isomorphic to the multiplicative group $\mathbb{G}_m$.
We put $G=D^\times/Z.$
For any place $\q\notin \ram(D)$ we fix an isomorphism $D_\q\cong M_2(F_\q)$ that in turn induces an isomorphism $G_\q\cong \PGL_2(F_\q).$
For any Archimedean place $\q\in \ram(D)$ we fix an isomorphism of $D_\q$ with the Hamilton quaternions, which yields an embedding $G_\q\into \PGL_2(\C)$.
In particular, we get an injection $j_\sigma\colon G(F)\subseteq G(F_\q)\xhookrightarrow{\sigma_\ast} \PGL_2(\C)$ for every embedding $\sigma\in \Sigma$ with underlying place $\q$.
We write
 \[
 j\colon G(F)\intoo \prod_{\sigma\in \Sigma} \PGL_2(\C)
 \]
for the diagonal embedding.

%Let $\n\subseteq \mathcal{O}_F$ be a non-zero ideal that is coprime to every prime in $\ram(B)$.
%For every prime $\q\notin \ram(B)$ of $F$ we define the compact open subgroup $K_\q(\n)\subseteq G_\q$ as the image of $\{g \in \GL_2(\mathcal{O}_\q)\mid g \mbox{ upper triangular mod } \n\}$ under the above chosen isomorphism.
%For $\q\in\ram(B)$ we define $K_\p(\n)\subseteq G_\p$ to be the unique maximal compact subgroup.
%If $S$ is a finite set of primes of $F$ we define
%$$K(\n)^{S}=\pro d_{\q\notin \Sigma}K_\p(\n)\subseteq G(\A\sinfty).$$
%If $\S$ is the empty set, we drop it from the notation.

%We decompose $\Sigma_\R= \Sigma_\R^{\ram} \cup \Sigma_\R^{\spl}$, where $\Sigma_\R^{\ram}$ denotes the set of those real embeddings of $F$ at which $B$ is ramified.
Let $S_\infty(D)$ be the set of all Archimedean places of $F$ at which $D$ is split.
We put
 \[
 q=\#S_\infty(D).
 \] 
Let $S_\R(D)\subseteq S_\infty(D)$ be the subset of real places.
We denote by $G_\infty^+$ the connected component of the identity of $G_\infty$ .
The group $\PGL_2(\C)$ and the units of the Hamilton quaternions are connected, whereas $\PGL_2(\R)$ has two connected components.
Therefore, we can identify
 \[
 \pi_0(G_\infty)=G_\infty/G_\infty^+\cong \{\pm 1\}^{S_\R(D)}
 \]
If $A\subseteq G_\infty$ is a subgroup, we put $A^+=A\cap G_\infty^+.$

\textit{An automorphic representation.}
Let $\pi^\prime=\otimes_\q \pi^\prime_{\q}$ be a cuspidal automorphic representation of $\PGL_2(\A)$ that is cohomological (see Section \ref{HarishChandra}).
If $F$ is totally real, then such automorphic representations (up to twists by the norm character) are in one-to-one correspondence with cuspidal Hilbert modular newforms with even weights and trivial Nebentypus.
 
We assume that the local component $\pi^\prime_\q$ is either a twist of the Steinberg representation or supercuspidal for all $\q$ dividing $\disc(D)$.
Thus, there exists a Jacquet--Langlands transfer $\pi$ of $\pi'$ to $G(\A)$, i.e.~an automorphic representation of $G(\A)$ such that $\pi^\prime_\q\cong \pi_\q$ for all places $\q\notin \ram(D).$
Moreover, we have that $\pi_\q$ is one-dimensional for all $\q \mid \disc(D)$ such that $\pi^\prime_\q$ is a twist of the Steinberg representation.

\section{Cohomology of p-arithmetic groups}
We recollect some basic facts about the cohomology of $\p$-arithmetic groups with values in duals of smooth representations.
\subsection{The Eichler--Shimura isomorphism}\label{cohomology}
In this section we recall how the representation $\pi$ contributes to the cohomology of the locally symmetric space attached to $G$.
\subsubsection{Weights and coefficient modules}\label{weights}
Let $k\geq 0$ be an even integer.
For any ring $R$ we let $V_{k}(R)\subseteq R[X,Y]$ be the space of homogeneous polynomials of degree $k$ with $\PGL_2(R)$-action given by
 \[
 (g.f)(X,Y)=\det(g)^{-k/2}f(bY+dX, aY+cX)\ \mbox{for}\ g=\begin{pmatrix} a & b \\ c& d \end{pmatrix}\in \PGL_2(R).
 \]
We may attach to $f\in V_{k}(R)$ an $R$-valued function on $G$ via
$\psi_f(g)=(g.f)(1,0).$
In case $b=\begin{pmatrix}b_1 & u\\ 0 & b_2 \end{pmatrix}$ is an upper triangular matrix, we have
\begin{align*}
\psi_f(bg)=(bg.f)(1,0)
&=(b.(g.f))(1,0)\\
&= b_1^{-k/2} b_2^{k/2} \cdot (g.f)(1,0)\\
&= b_1^{-k/2} b_2^{k/2} \cdot \psi_f(g)
\end{align*}
for all $g\in G$.

For a \emph{weight} $\bm{k}=(k_\sigma)_{\sigma\in\Sigma}\in 2\Z_{\geq 0}[\Sigma]$ we define the $\PGL_2(R)^{\Sigma}$-representation
 \[
 V_{{\bm{k}}}(R)= \otimes_{\sigma \in \Sigma} V_{k_\sigma}(R).
 \]
and $V_{{\bm{k}}}(R)^\vee$ as the $R$-linear dual of $V_{{\bm{k}}}(R).$

We may view $V_{\bm{k}}(\C)$ as a representation of $G(F)$ via the embedding $j$.
In fact, there exists a number field $E\subseteq \C$ such that every embedding $\sigma\colon F\into \C$ factors over $E$ and such that $G(F)$ acts on $V_{\bm{k}}(E)\subseteq V_{\bm{k}}(\C).$

\subsubsection{$(\mathfrak{g},K_\infty)$-cohomology}\label{HarishChandra}
Let $\mathfrak{g}$ denote the complexification of the Lie-Algebra of $G_\infty$.
We fix a maximal compact subgroup $K_\infty$ of $G_\infty$ with connected component 

We recall that $\pi$ is cohomological if and only if there exists a weight $\bm{k}=(k_\sigma)_{\sigma\in\Sigma}$ such that
 \[
 \HH^\bullet(\mathfrak{g},K_\infty^+,\pi_\infty \otimes V_{\bm{k}}(\C)^\vee)\neq 0.
 \]
See \cite{BW} for the notion of $(\mathfrak{g},K_\infty)$-cohomology.

The weight $\bm{k}$ is unique and we fix it from here on.
By \cite{Ha} (3.6.1) we know that 
\begin{align}\label{imagquad}
k_\sigma=k_{c\sigma}
\end{align}
holds for all $\sigma\in \Sigma$.
The group $\pi_0(G_\infty)\cong \pi_0(K_\infty)$ acts on $(\mathfrak{g},K_\infty)$-cohomology.
For every character $\epsilon\colon \pi_0(G_\infty)\to \{\pm 1\}$ we have the following dimension formulas for the $\epsilon$-isotypic component:
 \begin{align}\label{gK}
 \dim_\C\HH^{i}(\mathfrak{g},K_\infty^+,\pi_\infty \otimes V_{\bm{k}}^\vee)^\epsilon =\binom{\delta}{i-q}. 
 \end{align}
Via the Künneth theorem one may reduce the computation to that of the cohomology for each place $\q\mid \infty$ separately.
In case $G_\q$ is split, the computation is spelled out in Section (3.6.2) of \cite{Ha}.
The non-split case is trivial.

\subsubsection{Local systems}
For an open compact subgroup $K\subseteq G(\A^\infty)$ we define the locally symmetric space of level $K$ as
 \[
 \mathcal{X}_K=G(F)\backslash G(\A) / K K_\infty^+.
 \]
If $K$ is small enough, the topological space $\mathcal{X}_K$ carries the structure of a smooth manifold.

We fix a field $\Omega$ of characteristic zero.
Let $N$ be an $\Omega[G(F)]$-module and $\underline{N}$ the locally constant sheaf on $\mathcal{X}_K$ given by the fibres of the projection
 \[
 G(F)\backslash G(\A) \times N /K K_\infty^+ \ontoo \mathcal{X}_K,
 \]
where we let $G(F)$ act on $G(\A) \times N$ diagonally.

Right translation defines commuting actions of the component group $\pi_0(G_\infty)$ and the Hecke algebra $\T_K(\Omega)=\Omega[K\backslash G(\A^\infty)/K]$ of level $K$ on $N$-valued cohomology $\HH^\bullet(\mathcal{X}_K,\underline{N})$.

We assume in the following that $K$ is of the form $K=\prod_{\q}K_\q$ and that
$$(\pi^\infty)^K\neq 0.$$
Let $S$ be a finite set of primes of $F$ such that $S$ contains every $\q$ such that $K_\q$ is not maximal and put $K^S=\prod_{\q\notin S}K_\q$.
The Hecke algebra
$$\T_{K^S}^S(\Omega)=\Omega[K^S\backslash G(\A\sinfty)/K^S]$$
away from $S$ is central in $\T_K(\Omega)$.

We will assume for the rest of this article that $\pi^\infty$ has a model over $E$, which we denote $\pi_E^\infty.$
We may always assume this by enlarging $E$ slightly (see \cite{JanRat}, Theorem C).
If $\Omega$ is an extension of $E$, we put $\pi_\Omega^\infty=\pi_E^\infty \otimes_E \Omega.$
Let $\m_\pi^S\subseteq \T_{K^S}^S(\Omega)$ be the maximal ideal that is given by the kernel of the map
$\T_{K^S}^S(\Omega)\to \End_\Omega((\pi_\Omega^\infty)^K).$

 \begin{Thm}\label{dim}
 Let $\Omega$ be any extension of $E$.
 For every character $\epsilon\colon\pi_0(G_\infty)\to \{\pm 1\}$ we have
  \[\dim_\Omega
  \Hom_{\T_K(\Omega)}((\pi_\Omega^\infty)^K, \HH^i(\mathcal{X}_K,\underline{V_{\bm{k}}(\Omega)^\vee})^\epsilon)
  =\binom{\delta}{i-q}.
	\]
 Moreover, the localization $\HH^i(\mathcal{X}_K,\underline{V_{\bm{k}}(\Omega)^\vee})_{\m_\pi^S}$ is equal to the sum of the images of all homomorphisms from $(\pi_\Omega^\infty)^K$ to $\HH^i(\mathcal{X}_K,\underline{V_{\bm{k}}(\Omega)^\vee})$.
 \end{Thm}
  \begin{proof}
	One may deduce from the Borel-Serre compactification that the canonical map
	 \[
	\HH^i(\mathcal{X}_K,\underline{V_{\bm{k}}(E)^\vee})^\epsilon\otimes_E \Omega
	 \too  \HH^i(\mathcal{X}_K,\underline{V_{\bm{k}}(\Omega)^\vee})^\epsilon
	 \]
	is an isomorphism for every extensions $\Omega$ of $E$.
	Thus, we may reduce to the case $\Omega=\C$.
	In that case, the first claim follows from standard arguments about cohomological representations and equation \eqref{gK} (see for example Section III of \cite{Ha} for details in the case $G$ is split).
	
	The second claim follows from strong multiplicity one.
  \end{proof}

\subsubsection{Cohomology of arithmetic groups}
We are going to recast the above cohomology groups in terms of group cohomology.
Let $K\subseteq G(\A^\infty)$ be an open compact subgroup and $N$ an $\Omega[G(F)]$-module.
The group $G(F)$ acts on the space $\mathcal{F}(G(\A^\infty)/K,N)$ via $(\gamma.f)(g)=\gamma.f(\gamma^{-1}g)$ and the Hecke algebra $\T_K(\Omega)$ via right translation.
Thus we have commuting actions of the component group $\pi_0(G_\infty)$ and the Hecke algebra $\T_K(\Omega)=\Omega[K\backslash G(\A^\infty)/K]$ on the spaces
\[
 \HH^i(X_K,N):=\HH^i(G(F)^+,\mathcal{F}(G(\A^\infty)/K,N))
\]

A straightforward calculation shows the following:
 \begin{Lem}\label{coh1}
Suppose that $N$ is a $\Q$-vector space.
 There are canonical isomorphisms
  \[
	\HH^\bullet(X_K,N)\xlongrightarrow{\cong} \HH^\bullet(\mathcal{X}_K,\underline{N})
	\]
 that are equivariant with respect to the actions of the component group and the Hecke algebra.
 \end{Lem}

\subsection{Cohomology of $p$-arithmetic groups}\label{p-cohomology}
Let $\p$ be a prime of $F$ and $\Omega$ a field of characteristic zero.
Given an open compact subgroup $K^\p\subseteq G(\A\pinfty)$, an $\Omega[G_\p]$-module $M$ and an $\Omega[G(F)]$-module $N$
we let $G(F)$ act on the $\Omega$-vector space $\mathcal{F}(G(\A\pinfty)/K^\p,\Hom_\Omega(M,N))$ via $(\gamma.f)(g)(m)=\gamma.f(\gamma^{-1}g)(\gamma^{-1}(m))$
and put
\[
 \HH^i_{\Omega}(X^{\p}_{K^{\p}},M,N):=\HH^i(G(F)^+,\mathcal{F}(G(\A\pinfty)/K^\p,\Hom_\Omega(M,N))).
\]
Again one can define commuting actions of the component group $\pi_0(G_\infty)$ and the Hecke algebra 
$$\T^{\p}_{K^\p}(\Omega)=\Omega[K^\p\backslash G(\A\pinfty)/K^\p].$$

For later purposes we also define the following continuous variant:
let $A$ be an affinoid $\Q_p$-algebra.
Given a continuous $A$-module $M$ with a continuous $G_\p$-action and an $A[G(F)]$-module $N$ that is finitely generated and free over $A$ we put
\[
 \HH^i_{A,\cont}(X^{\p}_{K^{\p}},M,N):=\HH^i(G(F)^+,\mathcal{F}(G(\A\pinfty)/K^\p,\Hom_{A,\cont}(M,N))).
\]

Suppose that $M=M_1\otimes_\Omega M_2$ with both $M_1$ and $M_2$ being $\Omega[G_\p]$-modules.
Then by definition we have
\begin{align}\label{swap}
 \HH^i_{\Omega}(X^{\p}_{K^{\p}},M_1\otimes M_2,N)=\HH^i_{\Omega}(X^{\p}_{K^{\p}},M_1,\Hom_\Omega(M_2,N)),
\end{align}
where $G(F)$ acts on $M_2$ via the embedding $G(F)\into G_\p.$

Let us discuss an example of the above construction.
First we are going to recall the notion of compact induction:
let $K_\p\subset G_\p$ be an open compact subgroup and $L$ a $\Omega[K_\p]$-module.
The \emph{compact induction} $\cind_{K_\p}^{G_\p}L$ of $L$ to $G_\p$ is given by the set of functions $f\colon G_\p\to L$ that satisfy $f(gk) = k^{-1}.f(g)$ for all $g\in G_\p$, $k\in K_\p$ and have finite support modulo $K_\p$.
The group $G_\p$ acts on $\cind_{K_\p}^{G_\p}L$ via left translation.

If $L$ is a $\Omega[G_\p]$-module, the map
\begin{align}\label{cindisom}
L\otimes_\Omega \cind_{K_\p}^{G_\p}\Omega \too\cind_{K_\p}^{G_\p}L,\
l\otimes f\mapstoo [g\mapsto f(g)\cdot g^{-1}.l]
\end{align}
is a $G_\p$-equivariant isomorphism.
Thus in this case we get a canonical $\T^{\p}_{K^\p}(\Omega)$-equivariant isomorphisms
\begin{align}\label{relation}
\HH^i_{\Omega}(X^{\p}_{K^{\p}},\cind_{K_\p}^{G_\p}L,N)
\cong\HH^i(X_{K^{\p}\times K_\p},\Hom_\Omega(L,N)),
\end{align}
by \eqref{swap}.

We define
 \[
 \tilde{\HH}^i(X_{K^\p},N)=\varinjlim_{K_\p} \HH^i(X_{K^\p \times K_\p},N),
 \]
where the limit is taken over all open compact subgroups $K_\p\subseteq G_\p.$
This space carries commuting actions of $\pi_0(G_\infty)$, $G_\p$ and $\T^{\p}_{K^\p}(\Omega)$.
The goal of this section is to compare $\Hom_{\Omega[G_\p]}(M,\tilde{\HH}^i(X_{K^\p},N))$ and $\HH^i_{\Omega}(X^{\p}_{K^{\p}},M,N)$ in certain situations.

\begin{Lem}\label{ses}
Let $M$ be a smooth $\Omega[G_\p]$-representation of finite length.
Then, there exists a resolution
$$0\too P_1\too P_0\too M\too 0$$
where $P_i$, $i=0,1$, are finitely generated projective smooth representations.
\end{Lem}
\begin{proof}
In case $\p\in\ram(D)$ the group $G_\p$ is compact.
Thus, the category of smooth representations of $G_\p$ with $\Omega$-coefficients is semi-simple and $M$ is projective itself.

In case $D$ is split at $\p$, and therefore $G_\p\cong \PGL_2(F_\p)$, this is a consequence of the main theorem of \cite{SS}.
\end{proof}

Let $\pi\pinfty$ be the restricted tensor product of all local components of $\pi$ away from $\p$ and $\infty$.
This is a smooth irreducible representation of $G(\A\pinfty)$.
We fix models $\pi_E\pinfty$ respectively $\pi_{\p,E}$ of $\pi\pinfty$ respectively $\pi_\p$ over $E$ and a $G(\A^\infty)$-equivariant isomorphism
$$\pi\pinfty_E \otimes_E \pi_{\p,E}\cong \pi^\infty_E. $$
From now on $\Omega$ will always be an extension of $E$ and we put
$\pi_\Omega\pinfty=\pi_E\pinfty\otimes_E \Omega$ as well as  $\pi_{\p,\Omega}=\pi_{\p,E}\otimes_E \Omega$.
Further, we assume that $K^\p$ is of the form $K^{\p}=\prod_{\q\neq\p} K_\q$ and that $(\pi_\Omega\pinfty)^{K^{\p}}\neq 0$.
We fix a finite set $S$ of primes of $F$ as before such that $\p\in S.$
\begin{Pro}\label{smoothduals}
Let $M$ be a smooth $\Omega[G_\p]$-representation of finite length.
There are isomorphisms
$$\HH^q_{\Omega}(X^{\p}_{K^{\p}},M,V_{\bm{k}}(\Omega)^\vee)_{\m_\pi^S}\xlongrightarrow{\cong} \Hom_{\Omega[G_\p]}(M,\tilde{\HH}^q(X_{K^\p},V_{\bm{k}}(\Omega)^\vee)_{\m_\pi^S})$$
that are functorial in $M$ and equivariant under the actions of $\T^\p_{K^\p}(\Omega)$ and $\pi_0(G_\infty)$.
Furthermore, we have
\[
\dim_\Omega \HH^d_{\Omega}(X^{\p}_{K^{\p}},M,V_{\bm{k}}(\Omega)^\vee) < \infty
\]
for all $d \geq 0$.
\end{Pro}
\begin{proof}
If $M=P$ is projective, there are functorial isomorphisms
\begin{align}\label{proj}
\HH^d_{\Omega}(X^{\p}_{K^{\p}},P,V_{\bm{k}}(\Omega)^\vee)\xlongrightarrow{\cong} \Hom_{\Omega[G_\p]}(P,\tilde{\HH}^d(X_{K^\p},V_{\bm{k}}(\Omega)^\vee))
\end{align}
for all $d\geq 0$ by \cite{Ge4}, Lemma 3.5 (b).
In particular, we have 
$$\HH^d_{\Omega}(X^{\p}_{K^{\p}},P,V_{\bm{k}}(\Omega)^\vee)_{\m_\pi^{S}}=0$$
for all $d<q$ by Theorem \ref{dim}.

Thus the short exact sequence
$$0\too P_1\too P_0\too M\too 0$$
of Lemma \ref{ses} induces the exact sequence
$$0\to\HH^q_{\Omega}(X^{\p}_{K^{\p}},M,V_{\bm{k}}(\Omega)^\vee)_{\m_\pi^{S}}
\to \HH^q_{\Omega}(X^{\p}_{K^{\p}},P_0,V_{\bm{k}}(\Omega)^\vee)_{\m_\pi^{S}}
\to \HH^q_{\Omega}(X^{\p}_{K^{\p}},P_1,V_{\bm{k}}(\Omega)^\vee)_{\m_\pi^{S}}$$
and the first claim follows from the isomorphism \eqref{proj} for $P=P_0, P_1$.

The second claim follows similarly.
\end{proof}

\begin{Rem}\label{rem}
A typical projective smooth representation is the compact induction $\cind_{K_\p}^{G_\p}\Omega$ of the trivial representation from an open compact subgroup $K_\p\subseteq G_\p$.
In that case the result of \cite{Ge4} used above simply follows from the composition of isomorphisms
\begin{align*}
\HH^d_{\Omega}(X^{\p}_{K^{\p}},\cind_{K_\p}^{G_\p}\Omega,N)
\cong &\HH^d_{\Omega}(X_{K^{\p}\times K_\p},N)\\
\cong &\tilde{\HH}^d(X_{K^\p},N)^{K_\p}\\
\cong &\Hom_{\Omega[G_\p]}(\cind_{K_\p}^{G_\p}\Omega,\tilde{\HH}^d(X_{K^\p},N)^{K_\p}),
\end{align*}
where the first isomorphism is a special case of \eqref{relation}, the second follows from the Hochschild-Serre spectral sequence and the third from Frobenius reciprocity. 
\end{Rem}

Proposition \ref{smoothduals} together with Theorem \ref{dim} implies the following:
\begin{Cor}\label{irreducible}
Let $M$ be an irreducible smooth $\Omega[G_\p]$-representation.
Then
\[
\HH^q_{\Omega}(X^{\p}_{K^{\p}},M,V_{\bm{k}}(\Omega)^\vee)^{\epsilon}_{\m_\pi^{S}}=0
\]
unless $M\cong \pi_{\p,\Omega}.$
Furthermore, there is an isomorphism
\[
\HH^q_{\Omega}(X^{\p}_{K^{\p}},\pi_{\p,\Omega},V_{\bm{k}}(\Omega)^\vee)^{\epsilon}_{\m_\pi^{S}}\cong(\pi_\Omega\pinfty)^{K_\p}
\]
of $\T^{\p}_{K^\p}(\Omega)$-modules. In particular, it is an absolutely irreducible $\T^{\p}_{K^\p}(\Omega)$-module.
\end{Cor}

\begin{Rem}
The corollary above was implicitly proven in \cite{Sp} under the assumption that the local representation $\pi_\p$ has an Iwahori-fixed vector via computations with explicit resolution of $\pi_\p$ coming from the Bruhat--Tits tree.
\end{Rem}

\subsection{The Steinberg case}\label{steinberg}
We now assume that $D$ is split at $\p$.
We define the \emph{$\Omega$-valued smooth Steinberg representation} $\St_\p^\infty(\Omega)$ of $G_\p$ as the space of locally constant $\Omega$-valued functions on $\mathbb{P}^1(F_\p)$ modulo constant functions.
The group $G_\p\cong \PGL_2(F_\p)$ naturally acts on $\mathbb{P}^1(F_\p)$ and thus also on $\St_\p^\infty(\Omega).$

We assume for the moment that $\pi_\p=\St_\p^\infty(\C).$
Then by Corollary \ref{irreducible} we know that
$\HH^q_{\Omega}(X^{\p}_{K^{\p}},\St_\p^\infty(\Omega),V_{\bm{k}}(\Omega)^\vee)^{\epsilon}_{\m_\pi^{S}}$ is an irreducible $\T^{\p}_{K^\p}(\Omega)$-module.

Given smooth $\Omega$-representations $V$ and $W$ of $G_\p$ we denote by $\Ext^i_\infty(V,W)$ the Ext-groups in the category of smooth representations.
It is well known that
\begin{align}\begin{split}\label{one}
\dim_\Omega \Ext^i_\infty(\Omega,\St_\p^\infty(\Omega))=
\begin{cases} 1 & \mbox{if}\ i=1\\
0 & \text{otherwise}
\end{cases}
\end{split}
\end{align}
and
\begin{align}\begin{split}\label{two}
\dim_\Omega \Ext^i_\infty(\St_\p^\infty(\Omega),\St_\p^\infty(\Omega))=
\begin{cases} 1 & \mbox{if}\ i=0\\
0 & \text{otherwise}.
\end{cases}
\end{split}
\end{align}

The above calculations follow directly from the existence of the following two projective resolutions:
$$0\too \cind_{J_\p}^{G_\p}\chi_{-} \too \cind_{K_\p}^{G_\p} \Omega \too \Omega\too 0$$
and
$$0\too \cind_{K_\p}^{G_\p} \Omega\too \cind_{J_\p}^{G_\p} \chi_{-} \too \St_\p^\infty(\Omega)\too 0.$$
Here $K_\p=\PGL_2(\mathcal{O}_\p)$ is a maximal compact subgroup, $J_\p\subset G_\p$ is an open compact subgroup that contains an Iwahori subgroup as a normal subgroup and $\chi_{-}\colon J_\p\to \{\pm 1\}$ is a non-trivial character that is trivial on the Iwahori subgroup.
The exactness of the first sequence is just a reformulation of the contractibility of the Bruhat--Tits tree.
The second follows from the fact that the cohomology with compact support of the Bruhat--Tits tree is the Steinberg representation (see for example \cite{Sp}, equation (18)).
To keep with our promise not to use the Bruhat-Tits tree let us mention that alternative calculations of these Ext-groups that do not invoke the Bruhat--Tits tree - and work for more general reductive groups - can be found in \cite{Orlik} and \cite{Dat}.

Let $\mathcal{E}^\infty$ be the unique (up to multiplication by a scalar) extension of $\Omega$ by $\St_\p^\infty(\Omega)$, i.e., there exists a non-split exact sequence
\begin{align}\label{short}
0 \too \St_\p^\infty(\Omega) \too \mathcal{E}^\infty \too \Omega \too 0.
\end{align}
of $G_\p$-modules.

We denote by
$$c_{\infty}^\epsilon\colon\HH^q_{\Omega}(X^{\p}_{K^{\p}},\St_\p^\infty(\Omega),V_{\bm{k}}(\Omega)^\vee)^\epsilon_{\m_\pi^{S}}\too
\HH^{q+1}_{\Omega}(X^{\p}_{K^{\p}},\Omega,V_{\bm{k}}(\Omega)^\vee)^\epsilon_{\m_\pi^{S}}
$$
the (localization of the) boundary homomorphism of the long exact sequence induced by \eqref{short}.
It is a homomorphism of $\T^{\p}_{K^\p}(\Omega)$-modules.

The following generalization of \cite{Sp}, Lemma 6.2 (b) holds:
\begin{Lem}\label{ordiso}
The map
$$c_{\infty}^\epsilon\colon\HH^d_{\Omega}(X^{\p}_{K^{\p}},\St_\p^\infty(\Omega),V_{\bm{k}}(\Omega)^\vee)^\epsilon_{\m_\pi^{S}}\too
\HH^{d+1}_{\Omega}(X^{\p}_{K^{\p}},\Omega,V_{\bm{k}}(\Omega)^\vee)^\epsilon_{\m_\pi^{S}}
$$
is an isomorphism for every sign character $\epsilon$ and every $d\geq 0$.
\end{Lem}
\begin{proof}
We have to show that
$$\HH^{d}_{\Omega}(X^{\p}_{K^{\p}},\mathcal{E}^\infty,V_{\bm{k}}(\Omega)^\vee)^\epsilon_{\m_\pi^{S}}=0$$
for all $d\geq 0$.
Let
\begin{align}\label{projres}
0 \too P_1\too P_0\too \mathcal{E}^\infty \too 0
\end{align}
be a projective resolution of $\mathcal{E}^\infty$ as in Lemma \ref{ses}.

Again by \cite{Ge4}, Lemma 3.5 (b) we have
\begin{align*}
\HH^d_{\Omega}(X^{\p}_{K^{\p}},P_i,V_{\bm{k}}(\Omega)^\vee)^\epsilon_{\m_\pi^{S}}\xlongrightarrow{\cong} \Hom_{\Omega[G_\p]}(P_i,\tilde{\HH}^d(X_{K^\p},V_{\bm{k}}(\Omega)^\vee)^\epsilon_{\m_\pi^{S}})
\end{align*}
for all $d\geq 0$ and $i=0,1$.
As a $G_\p$-module $\tilde{\HH}^d(X_{K^\p},V_{\bm{k}}(\Omega)^\vee)^\epsilon_{\m_\pi^{S}}$ is isomorphic to some copies of $\St_\p^\infty(\Omega)$ by our assumption and Theorem \ref{dim}.

Thus, analyzing the long exact sequence induced from \eqref{projres} it is enough to show that
\[
\Ext^d_\infty(\mathcal{E}^\infty,\St_\p^\infty(\Omega))=0
\]
for all $d\geq 0$.
But this follows directly from applying $\Ext_{G_\p}(\cdot,\St_\p^\infty(\Omega) )$ to the short exact sequence \eqref{short} and the computations \eqref{one} and \eqref{two} of dimensions of smooth Ext-groups.
\end{proof}

Note that the approach we take here is the one of \cite{Ge4}, Lemma 3.7.

\section{Stabilizations}
We explain the connection between overconvergent cohomology and the cohomology of $\p$-arithmetic subgroups with values in duals of locally analytic principal series representations.
We fix a prime $\p$ of $F$, at which the quaternion algebra $D$ is split, and an embedding
\[
\iota_p\colon \overline{\Q}\intoo \overline{\Q_p},
\]
where $p$ denotes the rational prime underlying $\p$.
We define $E_p$ to be the completion of $E$ with respect to the topology induced by $\iota_p$.
\subsection{Smooth principal series}\label{ident}
We will give a representation-theoretic description of $\p$-stabilizations as in \cite{Ge}, Section 2.2.
Most of the basic results on smooth representations of $\PGL_2$ over local fields that are used in this section can be found in \cite{Bump}, Section 4.5.

As mentioned before, we may identify $G_\p$ with $\PGL_2(F_\p)$.
Let $B$ be the standard Borel subgroup of $G_\p$ of upper triangular matrices.

Given a smooth $\Omega$-representation $\tau$ of $B$ its \emph{smooth parabolic induction} is the space
$$i_B(\tau)=\{f\colon G_\p\to \tau\ \mbox{locally constant}\mid f(bg)=b.f(g)\ \forall b\in B,\ g\in G_\p \}.$$
The group $G_\p$ acts on $i_B(\tau)$ via right translation.

We identify the spaces of locally constant characters on $B$ with those on $F_\p^\times$ by mapping a locally constant character $\chi\colon F_\p^\times\too \Omega^{\times}$ to the character
\begin{align}\label{identification}
B^\times \too \Omega^{\times},\quad \begin{pmatrix}b_1 & u \\ 0 & b_2\end{pmatrix}\mapstoo \chi(b_1/b_2).
\end{align}

\begin{Def}
A $\p$-stabilization $(\chi,\vartheta)$ of $\pi_{\Omega}$ consists of a locally constant character $\chi\colon F_\p^\times\too \Omega^{\times}$ together with a non-zero $G_\p$-equivariant homomorphism
$$\vartheta\colon i_B(\chi)\too  \pi_{\p,\Omega}.$$
\end{Def}

Let $(\chi,\vartheta)$ be a $\p$-stabilization of $\pi_{\Omega}$.
Since $\pi_{\p,\Omega}$ is irreducible, $\vartheta$ is automatically surjective.
Moreover, $(\chi,\vartheta)$ induces a $\p$-stabilization of $\pi_{\Omega^\prime}$ for every extension $\Omega^\prime$ of $\Omega$.

By the classification of smooth irreducible representations of $G_\p$ we know that $\pi_\p$ is either supercuspidal or a quotient of a smooth parabolic induction as above (with $\Omega=\C$).
In the first case, $\pi_{\Omega}$ admits no stabilization.
In the later case, there always exists a finite extension $\Omega^\prime$ of $\Omega$ such that $\pi_{\Omega^\prime}$ admits a $\p$-stabilization $(\chi,\vartheta)$. Furthermore, the map $\vartheta$ is unique up to multiplication with a scalar.

The character $\chi$ is in general not unique:
Suppose $\pi_{\Omega^\prime}$ admits a $\p$-stabilization $(\chi,\vartheta)$.
If $\vartheta$ is an isomorphism, then $\chi^{2}\neq \cf$ and $\chi^{2}\neq |\cdot|_\p^2$.
Moreover, $i_B(\chi)$ is isomorphic to $i_B(\chi^{-1}|\cdot|_\p)$ but to no other principal series.
Thus, as long as $\chi^{2}\neq |\cdot|_\p$ there are two essentially different stabilizations.

If $\vartheta$ is not an isomorphism, then $\chi^{2}=\cf$ and the kernel of $\vartheta$ is one-dimensional generated by the function $g\mapsto \chi(\det(g)).$
Moreover, the sequence
$$0 \too \Ker(\vartheta)\too i_B(\chi) \too \pi_{\p,\Omega}\too 0$$
is non-split.
In particular, we have $\pi_{\p,\Omega}\cong \St_\p^\infty(\Omega)\otimes\chi$ in this case.

From the discussion above and Proposition \ref{smoothduals} we get the following:
\begin{Cor}\label{stabisom}
Suppose $\pi_{\Omega}$ admits a $\p$-stabilization $(\chi,\vartheta)$.
Then there is an isomorphism
\[
\HH^q_{\Omega}(X^{\p}_{K^{\p}},i_B(\chi),V_{\bm{k}}(\Omega)^\vee)^{\epsilon}_{\m_\pi^{S}}\cong(\pi_\Omega\pinfty)^{K_\p}
\]
of $\T_{K^\p}(\Omega)$-modules. In particular, it is an absolutely irreducible $\T_{K^\p}(\Omega)$-module.
\end{Cor}

\begin{Rem}
We will be mostly interested in the case that the representation $\pi_\p$ admits an invariant vector under the Iwahori group
$$I_\p=\{g\in \PGL_2(\mathcal{O}_\p) \mid g\ \mbox{is upper triangular} \bmod {\p}\}.$$
It is well known that having an Iwahori fixed vector is equivalent to having a stabilization $(\chi,\vartheta)$ with respect to an unramified character $\chi$, i.e., $\chi(\mathcal{O}_\p^\ast)=1.$
\end{Rem}

\subsection{Non-critical stabilizations}
Composing an embedding $\sigma\colon F\into E\subseteq  \overline{\Q}$ with $\iota_p$ induces a $p$-adic topology on $F$.
We define $\Sigma_\p\subseteq \Sigma$ to be the set of all embeddings inducing the topology coming from our chosen prime $\p$
 and put $\Sigma^\p=\Sigma\setminus \Sigma_\p$.
We identify $\Sigma_\p\subseteq \Sigma$ with the set of embeddings $F_\p\into E_p$.

Suppose $\Omega$ is a finite extension of $E_p$, which we will do from now on.
We may decompose
\[
V_{\bm{k}}(\Omega)=V_{\bm{k}_\p}(\Omega) \otimes_\Omega V_{\bm{k}^\p}(\Omega)
\]
with
\[
V_{\bm{k}_\p}(\Omega)=  \bigotimes_{\sigma \in \Sigma_\p} V_{k_\sigma}(\Omega)
\qquad \mbox{and}\qquad
V_{\bm{k}^\p}(\Omega)=\bigotimes_{\sigma \in \Sigma^\p} V_{k_\sigma}(\Omega).
\]
The representation $V_{\bm{k}_\p}(\Omega)$ of $G(F)$ extends to an algebraic representation of the group $G_\p$.

We will assume for the reminder of this section that $\pi_{\Omega}$ admits a $\p$-stabilization $(\chi,\vartheta)$.
We define the locally algebraic $G_\p$-representation
\[
i_B(\chi_{\bm{k}_\p})=i_B(\chi)\otimes_\Omega V_{\bm{k}_\p}(\Omega).
\]
Then, by \eqref{swap} we have a canonical $\T^{\p}_{K^\p}(\Omega)$-equivariant isomorphism
\[
\HH^d_{\Omega}(X^{\p}_{K^{\p}},i_B(\chi),V_{\bm{k}}(\Omega)^\vee)
\xlongrightarrow{\cong} \HH^d_{\Omega}(X^{\p}_{K^{\p}},i_B(\chi_{\bm{k}_\p}),V_{\bm{k}^\p}(\Omega)^\vee).
\]

Let $A$ be an affinoid $\Q_p$-algebra.
As in \ref{identification} we identify locally $\Q_p$-analytic characters from $F_\p^\times$ to $A^\times$ with those from $B$ to $A^\times$.
Given a locally analytic representation $\tau$ of $B$ its \emph{locally analytic parabolic induction} is given by
$$\II_B(\tau)=\{f\colon G_\p\to \tau\ \mbox{locally analytic}\mid f(bg)=b.f(g)\ \forall b\in B,\ g\in G_\p \}.$$
The group $G_\p$ acts on it via right translation.
Suppose that $A=\Omega$ and $\tau$ is finite-dimensional.
In that case the locally analytic parabolic induction is a strongly admissible locally analytic representation of $G_\p$.
The case of one-dimensional representations is Proposition 1.21 of \cite{EmUnitary}.
The proof works verbatim for finite-dimensional representations.

To any locally constant character $\chi\colon F_\p^\times\to \Omega^\times$ we associate the locally analytic character $\chi_{\bm{k}_\p}$ by
\[
\chi_{\bm{k}_\p}(x)=\chi(x)\prod_{\sigma\in \Sigma_\p} \sigma(x)^{-k_\sigma/2}.
\]
We may identify $i_B(\chi_{\bm{k}_\p})$ with a subspace of $\II_B(\chi_{\bm{k}_\p})$ via the map
\begin{align*}
\beta\colon i_B(\chi)\otimes \bigotimes_{\sigma\in\Sigma_\p} V_{k_{\sigma}}(\Omega)
&\too  \II_B(\chi_{\bm{k}_\p})\\
(f_{\infty}, (f_{\sigma})_{\sigma\in\Sigma_\p}) &\mapstoo f_\infty\cdot \prod_{\sigma\in\Sigma_\p} \psi_{f_\sigma}.
\end{align*}
See Section \ref{weights} for the definition of the functions $\psi_{f_\sigma}$.
\begin{Def}
A $\p$-stabilization $(\chi,\vartheta)$ of $\pi_{\Omega}$ is called non-critical if the canonical map
\[
\beta^\ast\colon\HH^d_{\Omega,\cont}(X^{\p}_{K^{\p}},\II_B(\chi_{\bm{k}_\p}),V_{\bm{k}^\p}(\Omega)^\vee)_{\m_\pi^S}
\too\HH^d_{\Omega}(X^{\p}_{K^{\p}},i_B(\chi_{\bm{k}_\p}),V_{\bm{k}^\p}(\Omega)^\vee)_{\m_\pi^S}
\]
is an isomorphism for all $d\geq 0$.
\end{Def}
Note that the notion of non-criticality depends on the level $K^\p$ away from $\p$ and the set $S$.
The notion gets stronger as smaller $K^\p$ and as bigger $S$ get.

\subsubsection{Locally algebraic and locally analytic Steinberg representation}
Assume for the moment that $\pi_\p$ is the Steinberg representation.
As mentioned above there is a unique $\p$-stabilization $\vartheta\colon i_B(\cf)\to \St^\infty_\p(\Omega)$ which has a one dimensional kernel.
We say $\pi$ is non-critical at $\p$ if this unique $\p$-stabilization is non-critical.

We define the \emph{locally algebraic Steinberg representation of weight $\bm{k}_\p$} via
$$\St_{\bm{k}_\p}^{\infty}(\Omega)=\St_{\p}^{\infty}(\Omega)\otimes_{\Omega} V_{\bm{k}_\p}(\Omega)$$
and the \emph{locally analytic Steinberg representation of weight $\bm{k}_\p$} as the quotient
$$\St_{\bm{k}_\p}^{\an}(\Omega)=\II_B(\chi_{\bm{k}_\p})/V_{\bm{k}_\p}(\Omega).$$
Thus, we have a natural embedding $$\kappa\colon\St_{\bm{k}_\p}^{\infty}(\Omega)\into \St_{\bm{k}_\p}^{\an}(\Omega).$$

Unravelling the definitions we get the following statement.
\begin{Pro}\label{clear}
Suppose that $\pi_{\p}\cong\St_\p(\C)$ and $\pi$ is non-critical.
Then the canonical map
\[
\kappa^\ast\colon\HH^d_{\Omega,\cont}(X^{\p}_{K^{\p}},\St_{\bm{k}_\p}^{\an}(\Omega),V_{\bm{k}^\p}(\Omega)^\vee)_{\m_\pi^S}
\too\HH^d_{\Omega}(X^{\p}_{K^{\p}},\St_{\bm{k}_\p}^{\infty}(\Omega),V_{\bm{k}^\p}(\Omega)^\vee)_{\m_\pi^S}
\]
is an isomorphism for all $d\geq 0$.
\end{Pro}

\begin{Rem}\label{CMPremark}
On page 653 of \cite{CMP} it is claimed that a property closely related to non-criticality always holds if the quaternion algebra $D$ is totally definite.
It is alluded to an Amice--V\'elu and Vishik-type argument.
But to the knowledge of the authors of this article the most general results of that type are in Section 7 of \cite{breuil_2011}, which essentially only cover the case of non-critical slope.
\end{Rem}

In the following we are going to show that if the representation $\pi_\p$ has an Iwahori-fixed vector the above definition of non-criticality is equivalent to the one given in terms of overconvergent cohomology that is used for example in \cite{BDJ} or \cite{BH}.
In particular, the classicality theorem for overconvergent cohomology will give a numerical criterion for the non-criticality of a stabilization.
In order to state this criterion later we will need the following definition.
\begin{Def}
Let $(\chi,\vartheta)$ be a $\p$-stabilization of $\pi_{\Omega}$.
The $p$-adic valuation of
$\prod_{\sigma\in\Sigma_\p}\sigma(\varpi_\p)^{\frac{k_\sigma}{2}}\chi(\omega_\p)$
is called the slope of $(\chi,\vartheta).$
We say that $(\chi,\vartheta)$ has non-critical slope if its slope is less than $\frac{1}{e_\p}\min_{\sigma\in\Sigma_\p}(k_\sigma+1)$.
\end{Def}

\subsection{Overconvergent cohomology}\label{overconvergent}
We give a quick overview over the basics of overconvergent cohomology.
\subsubsection{Locally analytic inductions} Let
\[
I_\p^n=\{g\in \PGL_2(\mathcal{O}_\p) \mid g\ \mbox{is upper triangular} \bmod {\p^n}\}.
\]
In particular, $I_\p=I_\p^{1}$ is the standard Iwahori subgroup.

Let $A$ be an affinoid $\Q_p$-algebra.
Let $\chi\colon B\cap I_\p\rightarrow A^\times $ be a locally analytic character. This means that there exists a minimal integer $n_\chi\geq 1$ such that $\chi$ restricted to $B\cap I_\p^{n_\chi}$ is analytic.
For any integer $n\geq n_\chi$ define the $A[I_\p]$-module $\mathcal{A}_\chi^n$ of functions $f\colon I_\p\longrightarrow A$ such that
\begin{itemize}
\item $f\ \mbox{is analytic on any coset of}\ I_\p/I_\p^n$,
\item $f(bk)=\chi(b)f(k)\quad \forall b\in B\cap I_\p,\ k\in I_\p$
\end{itemize}
and put
$$\mathcal{A}_\chi=\bigcup_{n\geq n_\chi}\mathcal{A}_\chi^n.$$
The $A[I_\p]$-module $\mathcal{A}_\chi$ is the locally analytic induction of $\chi$ to $I_\p$.
In the special case that $A=\Omega$ is a finite extension of $E_p$ and $\chi=\cf_{\bm{k}_\p}$ we put
$$\mathcal{A}_{\bm{k}_\p}=\mathcal{A}_{\cf_{\bm{k}_\p}}.$$

The Iwahori decomposition gives an isomorphism $I_\p\cong (I_\p\cap\overline{\mathbf{N}})\times B(\mathcal{O}_{F_\p})$, where $\overline{\mathbf{N}}$ denotes the group of unipotent lower triangular matrices.
Thus, restricting a function $f\in\mathcal{A}_\chi$ to $I_\p\cap\overline{\mathbf{N}}$ induces an isomorphism between $\mathcal{A}_\chi$ and the space $\mathcal{A}(I_\p\cap\overline{\mathbf{N}},A)$ of locally analytic $A$-valued functions on $I_\p\cap\overline{\mathbf{N}}$.
An analogous bijection holds between $\mathcal{A}_\chi^n$ and $\mathcal{A}_n(I_\p\cap\overline{\mathbf{N}},A)$, the space of $n$-locally analytic functions on $I_\p\cap\overline{\mathbf{N}}$.

\subsubsection{The $U_\p$ operator}
Consider the compact induction $\cind_{I_\p}^{G_\p}(\mathcal{A}_{\chi}^n)$.
By Frobenius reciprocity, the ring $\End_{A[G_\p]}(\cind_{I_\p}^{G_\p}(\mathcal{A}_{\chi}^n))$ can be identified with the space of all functions $\Psi\colon G_\p\too \End_A(\mathcal{A}_\chi^n)$ such that
\begin{itemize}
\item $\Psi$ is $I_\p$-biequivariant, that is $\Psi(k_1 g k_2)=k_1\Psi(g) k_2$ in $\End_A(\mathcal{A}_\chi^n)$, for all $k_1,k_2\in I_\p$, $g\in G_\p$, and
\item for any element $f\in \mathcal{A}_\chi^n$, the function $G_\p\to \mathcal{A}_\chi^n$, $g\mapsto \Psi(g)(f)$ is compactly supported.
\end{itemize}
Let
$u_\p:=\begin{pmatrix}
\varpi_\p & 0\\
0 & 1
\end{pmatrix}$.
Consider the element $\varphi_{u_\p}\in \End_A(\mathcal{A}_\chi^n)$ defined by
$$\varphi_{u_\p}(f)(\overline{n})=f(u_\p\overline{n} u_\p^{-1})\qquad \text{for all}\ f\in \mathcal{A}_\chi^n,\ \overline{n}\in I_\p\cap \overline{\mathbf{N}}.$$
By \cite[Lemma 2.2]{KoSch}, there exists a unique $I_\p$-biequivariant function $\Psi_{u_\p}\colon G_\p\to \End_A(\mathcal{A}_\chi^n)$ such that $\text{supp}(\Psi_{u_\p})=I_\p u_\p^{-1} I_\p$ and $\Psi_{u_\p}(u_\p^{-1})=\varphi_{u_\p}$.

Abusing notation, we will simply denote by $u_\p$ the $G_\p$-equivariant endomorphism of $\cind_{I_\p}^{G_\p}(\mathcal{A}_\chi^n)$ corresponding to $\Psi_{u_\p}$.

For $d\geq 0$ we set
$$
\HH^d(X_{K^{\p}\times I_{\p}},\mathcal{D}_{\chi,\bm{k}^\p}^n)=\HH^d_{A,\cont}(X^{\p}_{K^{\p}},\cind_{I_\p}^{G_\p}(\mathcal{A}_\chi^n),A\otimes_\Omega V_{\bm{k}^\p}(\Omega)^\vee).
$$
For now, we are mostly interested in the following special case:
$A=\Omega$ is a finite extension of $E_p$ and $\chi=\cf_{\bm{k}_\p}$.
In this situation we abbreviate
$$\HH^d(X_{K^{\p}\times I_{\p}},\mathcal{D}_{\bm{k}}^n)=\HH^d(X_{K^{\p}\times I_{\p}},\mathcal{D}_{\cf_{\bm{k}_\p},\bm{k}^\p}^n).$$
Later we will also need the case that $A$ is the coordinate ring of an affinoid subspace of the weight space (see Section \ref{weightspace}).

\begin{Rem}
In \cite{BDJ} overconvergent cohomology groups are introduced depending on a subset of the set of all primes of $F$ lying above $p$.
The spaces defined above correspond to the subset consisting only of the prime $\p$.
\end{Rem}

The endomorphism $u_\p$ induces an operator, that we denote by $U_\p^{\circ}$, on cohomology:
$$U_\p^{\circ}\colon H^d(X_{K^\p\times I_\p}, \mathcal{D}_{\chi,\bm{k}^\p}^n)\too H^d(X_{K^\p\times I_\p}, \mathcal{D}_{\chi,\bm{k}^\p}^n).$$

Similar as before we may identify $V_{\bm{k}_\p}(\Omega)$ with the space of (globally) algebraic vectors in $\mathcal{A}_{\bm{k}_\p}^n$.
It induces an embedding
\[
\cind_{I_\p}^{G_\p}(V_{\bm{k}_\p}(\Omega))\too\cind_{I_\p}^{G_\p}(\mathcal{A}_{\bm{k}_\p}^n)
\]
and the subspace $\cind_{I_\p}^{G_\p}(V_{\bm{k}_\p}(\Omega))$ is clearly invariant under the action of $u_\p$.

Thus by invoking \eqref{relation} we get a $\mathbb{T}^\p_{K^\p}(\Omega)$-equivariant map
\begin{equation}\label{OC}
\HH^d(X_{K^\p\times I_\p},\mathcal{D}_{\bm{k}}^n)\longrightarrow \HH^d(X_{K^\p\times I_\p},V_{\bm{k}}^\vee)
\end{equation}
in cohomology.
We denote the natural operator on the right hand side induced by $u_\p$ by $U_\p$.
The map $\eqref{OC}$ intertwines the action of the Hecke operator $U_\p^{\circ}$ on $H^d(X_{K^\p\times I_\p},\mathcal{D}_{\bm{k}}^n)$ with the action of $\prod_{\sigma\in\Sigma_\p}\sigma(\varpi_\p)^{\frac{k_\sigma}{2}} U_\p$ on $H_c^d(X_{K^\p\times I_\p},V_{\bm{k}}^\vee)$.
This follows from a simple analysis of the change of action of $u_\p$ under the isomorphism
$$\cind_{I_\p}^{G_\p} V_{\bm{k}_\p}\cong  V_{\bm{k}_\p}\otimes_{\Omega}\cind_{I_\p}^{G_\p}\Omega $$
given by \eqref{cindisom}.
We define the Hecke operator $U_\p^{\circ}$ on the right hand side of \eqref{OC} by
$$U_\p^{\circ}=\prod_{\sigma\in\Sigma_\p}\sigma(\varpi_\p)^{\frac{k_\sigma}{2}} U_\p$$
and similarly we define an action of $U_p$ on the left hand side of $\eqref{OC}$.

\subsubsection{Slope decompositions and classicality} 
We give a reminder on slope decompositions.
As before, $A$ denotes a affinoid $\Q_p$-algebra.
Let $M$ be an $A$-module equipped with an $A$-linear endomorphism $u\colon M \to M$.
Fix a rational number $h\geq 0$.
A polynomial $Q\in A[x]$ is multiplicative of slope $\leq h$ if
\begin{itemize}
\item the leading coefficient of $Q$ is a unit in $A$ and 
\item every edge of the Newton polygon of $Q$ has slope $\leq h$.
\end{itemize}
We put $Q^\ast(x)= x^{\deg Q} Q(1/x).$
An element $m\in M$ is said to be of slope $\leq h$ if there is a multiplicative polynomial $Q\in A[x]$ of slope $\leq h$
such that $Q^{\ast}(u) m=0$.
Let $M^{\leq h} \subseteq M$ be the submodule of elements of M of slope $\leq h$.
\begin{Def}
A slope $\leq h$ decomposition of $M$ is an $A[u]$-module isomorphism
$$M \cong M^{\leq h} \oplus M^{> h}$$
such that 
\begin{itemize}
\item $M^{\leq h}$ is a finitely generated $A$-module and
\item $Q^\ast(u)$ acts invertibly on $M^{>h}$ for every multiplicative polynomial $Q\in A[x]$ of slope $\leq h$.
\end{itemize}
\end{Def}
Note that if $A=\Omega$ is a finite extension of $E_p$ and $M$ is finite-dimensional, then a slope $\leq h$ decomposition always exists.
If $M$ admits a slope $\leq h$ decomposition for all $h\geq 0$, we put
$$M^{< \infty}=\bigcup_{h\geq 0} M^{\leq h}.$$

The most remarkable result about slope decomposition is the following theorem which was first proved by Ash and Stevens \cite{AshStevens} in special cases over $\Q$ and then generalized by Urban \cite{Urban} and Hansen \cite{Hansen} to more general settings.
See also \cite{BWi0} for a detailed treatment if the case of $\PGL_2$ over arbitrary number fields.
In these results one always considers all primes above $\p$ simultaneously.
The modifications necessary to allow subsets of all primes above $\p$ are explained in the proof of \cite{BDJ}, Theorem 2.7.

\begin{Thm}\label{AS}
For every $d\geq 0$ and every $h\geq 0$ the cohomology groups
$$H^d(X_{K^\p\times I_\p},\mathcal{D}^n_{\bm{k}})$$
admit a slope $\leq h$ decomposition with respect to the Hecke operator $U_\p^{\circ}$.

If $h<\frac{1}{e_\p}\min_{\sigma\in\Sigma_\p}(k_\sigma+1)$, where $e_\p$ is the ramification index of $\p$, then for all $d\geq 0$ the map \eqref{OC} induces the following $\mathbb{T}^\p_{K^\p}$-equivariant isomorphism
$$H^d(X_{K^\p\times I_\p},\mathcal{D}^n_{\bm{k}})^{\epsilon,\leq h}\overset{\sim}{\longrightarrow }H^d(X_{K^\p\times I_\p},V_{\bm{k}}^\vee)^{\epsilon,\leq h}.$$
Here the slope decomposition is taken with respect to $U_\p^{\circ}$ on both sides.
\end{Thm}

\subsection{Overconvergent cohomology and non-critical stabilization}\label{overconvergentsec}

Let $A$ be an affinoid algebra and $\chi\colon F_\p^\times\too A^\times$ be a locally analytic character. Denote by $\chi_0$ its restriction to $\mathcal{O}_{F_\p}^\times$.
An element $f\in \mathcal{A}_{\chi_0}^{n}$ can uniquely extended to a function on $B I_\p \subset G_\p$ by putting
$f(bk)=\chi(b)f(k).$
Since $B I_\p \subset G_\p$ is open, extension by zero yields an $I_\p$-equivariant $A$-linear injection
$$\mathcal{A}_{\chi_0}^{n}\intoo\mathbb{I}_B(\chi)\vert_{I_\p}.$$

By Frobenius reciprocity, it induces a unique $G_\p$-equivariant $A$-linear morphism
\begin{equation}\label{aug}
\mathrm{aug}_\chi\colon \cind_{I_\p}^{G_\p}(\mathcal{A}_{\chi_0}^n)\longrightarrow \mathbb{I}_B(\chi).
\end{equation}

The following theorem is due to Kohlhaase and Schraen.
\begin{Thm}\label{Koszul}
For $n\geq n_{\chi_0}$, the short sequence
\begin{equation}\label{exact}
0\longrightarrow \cind_{I_\p}^{G_\p}(\mathcal{A}_{\chi_0}^n)\xlongrightarrow{u_\p-\chi(\varpi_\p)} \cind_{I_\p}^{G_\p}(\mathcal{A}_{\chi_0}^n)\xlongrightarrow{\mathrm{aug}_\chi} \II_B(\chi)\longrightarrow 0
\end{equation}
is exact, where $\mathrm{aug}_\chi$ is the map obtained by composing \eqref{aug} with the natural map $\cind_{I_\p}^{G_\p}(\mathcal{A}_{\chi_0}^n)\rightarrow \cind_{I_\p}^{G_\p}(\mathcal{A}_{\chi_0})$.
\end{Thm}

\begin{proof}
See \cite[Proposition 2.4 and Theorem 2.5]{KoSch}.
\end{proof}

Let $\chi\colon F_\p^\times \to\Omega^\times$ be a smooth character.
Similarly as above, the sequence
\begin{align*}
0\longrightarrow \cind_{I_\p}^{G_\p}(\Omega)\xlongrightarrow{u_\p-\chi(\varpi_\p)} \cind_{I_\p}^{G_\p}(\Omega)\xlongrightarrow{\mathrm{aug}_\chi} i_B(\chi)\longrightarrow 0
\end{align*}
is exact.
This can be deduced from Borel's theorem that $\cind_{I_\p}^{G_\p}(\Omega)$ is a flat module over the Iwahori-Hecke algebra (see the end of Section 3.1 of \cite{GR}).
Tensoring the above exact sequence with $V_{\bm{k}_\p}$ and using \eqref{cindisom} we get a short exact sequence
\begin{align}\label{smoothup}
0\longrightarrow \cind_{I_\p}^{G_\p}(V_{\bm{k}_\p}) \xlongrightarrow{u_\p-\chi_{\bm{k}_\p}(\varpi_\p)} \cind_{I_\p}^{G_\p}(V_{\bm{k}_\p})\xlongrightarrow{\mathrm{aug}_{\chi_{\bm{k}_\p}}} i_B(\chi_{\bm{k}_\p})\longrightarrow 0.
\end{align}

Given a $\p$-stabilization $(\chi,\vartheta)$ of $\pi_\Omega$ we define
$$\m_{\pi,(\chi,\vartheta)}^S\subseteq \T_K^S(\Omega)[U_\p]$$
to be the maximal ideal generated by $\m_{\pi}^S$ and $U_\p-\chi(\varpi_\p).$
In accordance with \cite{BDJ}, Definition 2.12, and \cite{BH}, Definition 1.5.1, we make the following definition:
\begin{Def}
The maximal ideal $\m_{\pi,(\chi,\vartheta)}^S\subseteq \T_K^S(\Omega)[U_\p]$ is non-critical if the map
\[
\HH^d(X_{K^\p\times I_\p},\mathcal{D}_{\bm{k}}^n)^{< \infty}_{\m_{\pi,(\chi,\vartheta)}^S}\longrightarrow \HH^d(X_{K^\p\times I_\p},V_{\bm{k}}^\vee)_{\m_{\pi,(\chi,\vartheta)}^S}
\]
induced by \eqref{OC} is isomorphisms for all $d\geq 0$.
\end{Def}

\begin{Pro}\label{equivalence}
Suppose that $\pi_\p$ has an Iwahori fixed vector and let $(\chi,\vartheta)$ be a $\p$-stabilization of $\pi_\Omega$.
Then the following are equivalent:
\begin{enumerate}[(i)]
\item $(\chi,\vartheta)$ is non-critical
\item $\m_{\pi,(\chi,\vartheta)}^S$ is non-critical.
\end{enumerate}
\end{Pro}
\begin{proof}
The long exact sequences induced by \eqref{exact} and \eqref{smoothup} yield the following diagram with exact columns:
\begin{center}
 \begin{tikzpicture}
    \path
		(0,-1) node[name=B]{$\HH^d_{\Omega,\cont}(X_{K^\p}^\p,\II_B(\chi_{\bm{k}_\p}),V_{\bm{k}^\p}(\Omega)^\vee)_{\m_{\pi}^S}$}
		(0,-2.6) node[name=C]{$\HH^d(X_{K^\p\times I_\p},\mathcal{D}_{\bm{k}}^n)_{\m_{\pi}^S}$}
		(0,-4.2) node[name=D]{$\HH^d(X_{K^\p\times I_\p},\mathcal{D}_{\bm{k}}^n)_{\m_{\pi}^S}$}
		(0,-5.8) node[name=E]{$\HH^{d+1}_{\Omega,\cont}(X_{K^\p}^\p,\II_B(\chi_{\bm{k}_\p}),V_{\bm{k}^\p}(\Omega)^\vee)_{\m_{\pi}^S}$}
		(6.5,-1) node[name=G]{$\HH^d_{\Omega}(X_{K^\p}^\p,i_B(\chi_{\bm{k}_\p}),V_{\bm{k}^\p}(\Omega)^\vee)_{\m_{\pi}^S}$}
		(6.5,-2.6) node[name=H]{$\HH^d(X_{K^\p\times I_\p},V_{\bm{k}}(\Omega)^\vee)_{\m_{\pi}^S}$}
		(6.5,-4.2) node[name=I]{$\HH^d(X_{K^\p\times I_\p},V_{\bm{k}}(\Omega)^\vee)_{\m_{\pi}^S}$}
		(6.5,-5.8) node[name=J]{$\HH^{d+1}_{\Omega}(X_{K^\p}^\p,i_B(\chi_{\bm{k}_\p}),V_{\bm{k}^\p}(\Omega)^\vee)_{\m_{\pi}^S}$};
		\draw[->] (B) -- (C) node[midway, left]{$\mathrm{aug}_{\chi_{\bm{k}_\p}}^\ast$};
		\draw[->] (C) -- (D) node[midway, left]{$U_\p^{\circ}-\chi_{\bm{k}_\p}(\varpi_\p)$};
		\draw[->] (D) -- (E) node[midway, left]{$\partial$};
		\draw[->] (G) -- (H) node[midway, left]{$\mathrm{aug}_{\chi_{\bm{k}_\p}}^\ast$};
		\draw[->] (H) -- (I) node[midway, right]{$U_\p^{\circ}-\chi_{\bm{k}_\p}(\varpi_\p)$};
		\draw[->] (I) -- (J) node[midway, left]{$\partial$};
    \draw[->] (B) -- (G) node[midway, above]{$\beta^\ast$};
    \draw[->] (C) -- (H) node[midway, above]{$\eqref{OC}$};
    \draw[->] (D) -- (I) node[midway, above]{$\eqref{OC}$};
		\draw[->] (E) -- (J) node[midway, above]{$\beta^\ast$};
  \end{tikzpicture} 
\end{center}
From the existence of slope decompositions (see Theorem \ref{AS}) we may replace $\HH^d(X_{K^\p\times I_\p},\mathcal{D}_{\bm{k}}^n)_{\m_{\pi}^S}$ by $\HH^d(X_{K^\p\times I_\p},\mathcal{D}_{\bm{k}}^n)_{\m_{\pi,(\chi,\vartheta)}^S}$ in diagram above (and similarly for cohomology with coefficients in $V_{\bm{k}}$).
The claim then follows by induction on $d$.
\end{proof}

Applying the second part of Theorem \ref{AS} we get the following: 
\begin{Cor}\
Suppose that $\pi_\p$ has an Iwahori-fixed vector.
If a $\p$-stabilization $(\chi,\vartheta)$ of $\pi_\Omega$ has non-critical slope, then it is non-critical.
\end{Cor}

Suppose $\pi_\p=\St_\p(\C).$
Then the corollary above shows that $\pi$ is non-critical at $\p$ if
\begin{enumerate}[(i)]
\item $k_\sigma=0$ for all $\sigma\in\Sigma_\p$ or
\item $F_\p=\Q_p$ or
\item $[F_\p:\Q_p]=2$ and $k_{\sigma_1}=k_{\sigma_2}$ where $\Sigma_\p=\{\sigma_1,\sigma_2\}.$ 
\end{enumerate}
In particular, this holds if $F=\Q$ or if $F$ is imaginary quadratic by \eqref{imagquad}.

\section{Automorphic L-invariants}
The main aim of this section is to define automorphic $\LI$-invariants for the representation $\pi$ under the assumption that the local component of $\pi$ at a prime $\p$ is Steinberg.
\subsection{Extensions of locally analytic Steinberg representations}\label{sec-extensions}
The following construction of extensions is due to Breuil (see \cite{Br}, Section 2.1).
Let $\lambda\colon F_\p^{\times}\to\Omega$ be a continuous homomorphism.
Note that $\lambda$ is automatically locally $\Q_p$-analytic.
We define $\tau_\lambda$ to be the two dimensional $\Omega$-representation of $B$ given by
$$\begin{pmatrix} a & u \\ 0 & d \end{pmatrix} \mapstoo \begin{pmatrix} 1 & \lambda(a/d) \\ 0 & 1 \end{pmatrix} $$
and put $\tau_{\lambda,\bm{k}_\p}=\tau_\lambda\otimes \chi_{\bm{k}_\p}$.
The short exact sequence
$$0 \too\chi_{\bm{k}_\p}\too \tau_{\lambda,\bm{k}_\p} \too \chi_{\bm{k}_\p}\too 0$$
induces the short exact sequence
$$0 \too \II_B(\chi_{\bm{k}_\p})\too \II_B(\tau_{\lambda,\bm{k}_\p})\too \II_B(\chi_{\bm{k}_\p})\too 0 $$
of locally analytic representations.
Pullback via $V_{\bm{k}_\p}(\Omega)\into \II_B(\chi_{\bm{k}_\p})$ and pushforward along $\II_B(\chi_{\bm{k}_\p})\onto \St_{\bm{k}_\p}^{\an}(\Omega)$ yields the exact sequence
\begin{align}\label{Steinbergclasses}
0 \too \St_{\bm{k}_\p}^{\an}(\Omega)\too \mathcal{E}_{\lambda,\bm{k}_\p}\too V_{\bm{k}_\p}(\Omega)\too 0.
\end{align}

\begin{Rem}
Given two locally $\Q_p$-analytic $\Omega$-representations $W_1$ and $W_2$ we denote by $\Ext^{1}_{\an}(W_1,W_2)$ the space of locally $\Q_p$-analytic extensions of $W_2$ by $W_1$. The map
$$\Hom_{\cont}(F_\p^{\times},\Omega)\too \Ext^{1}_{\an}(V_{\bm{k}_\p}(\Omega),\St_{\bm{k}_\p}^{\an}(\Omega)),\ \lambda \mapstoo \mathcal{E}_{\lambda,\bm{k}_\p}$$
is an isomorphism.
In the case $F_\p=\Q_p$ this is due to Breuil.
In fact, an analogous statement is true for more general split reductive groups (see \cite{Ding}, Theorem 1, and \cite{Ge4}, Theorem 2.15).
\end{Rem}

We put
$$\mathcal{E}_{\bm{k}_\p}^\infty= \mathcal{E}^\infty\otimes_\Omega V_{\bm{k}_\p}$$
where $\mathcal{E}^\infty$ is the smooth extension of \eqref{short}.
By definition we have an exact sequence
\begin{align*}
0 \too \St_{\bm{k}_\p}^\infty(\Omega) \too \mathcal{E}_{\bm{k}_\p}^\infty \too V_{\bm{k}_\p} \too 0.
\end{align*}
It is easy to see that the extension $\mathcal{E}_{\bm{k}_\p}^\infty$ is mapped to a multiple of $\mathcal{E}_{\ord_\p,\bm{k}_\p}$ under the map $\Ext^{1}_{\an}(V_{\bm{k}_\p}(\Omega),\St_{\bm{k}_\p}^{\infty}(\Omega))\to \Ext^{1}_{\an}(V_{\bm{k}_\p}(\Omega),\St_{\bm{k}_\p}^{\an}(\Omega))$.

\subsection{Definition of the L-invariant}\label{definition}
We assume for the rest of this article that $\pi_\p$ is the Steinberg representation.

For a continuous homomorphism $\lambda\colon F_\p^\times\to \Omega$ we let
$$c_{\lambda}^\epsilon\colon\HH^q_{\Omega,\cont}(X^{\p}_{K^{\p}},\St_{\bm{k}_\p}^{\an}(\Omega),V_{\bm{k}^\p}(\Omega)^\vee)^\epsilon_{\m_\pi^{S}}\too
\HH^{q+1}_{\Omega}(X^{\p}_{K^{\p}},V_{\bm{k}_\p}(\Omega),V_{\bm{k}^\p}(\Omega)^\vee)^\epsilon_{\m_\pi^{S}}
$$
be the boundary map induced by the dual of the short exact sequence \eqref{Steinbergclasses}.
This map is clearly a $\T^{\p}_{K^\p}(\Omega)$-module homomorphism.

\begin{Def}\label{defdef}
The automorphic $\LI$-invariant of $\pi$ at $\p$
$$\LI_\p(\pi)^\epsilon\subseteq \Hom_{\cont}(F_\p^\times,\Omega)$$
is the kernel of the map $\lambda\mapsto c_{\lambda}^\epsilon.$ 
\end{Def}

\begin{Pro}
Assume that $\pi$ is non-critical at $\p$.
Then the codimension of the $\LI$-invariant $\LI_\p(\pi)^\epsilon\subseteq \Hom_{\cont}(F_\p^\times,\Omega)$ is equal to one.
Moreover, it does not contain the space of locally constant characters.
\end{Pro}
\begin{proof}
Combining Proposition \ref{clear} and Lemma \ref{ordiso} we see that the space of $\T_{K^\p}(\Omega)$-linear homomorphisms between the two modules $\HH^q_{\Omega,\cont}(X^{\p}_{K^{\p}},\St_{\bm{k}_\p}^{\an}(\Omega),V_{\bm{k}^\p}(\Omega)^\vee)^\epsilon_{\m_\pi^{S}}$ and $\HH^{q+1}_{\Omega}(X^{\p}_{K^{\p}},V_{\bm{k}_\p}(\Omega),V_{\bm{k}^\p}(\Omega)^\vee)^\epsilon_{\m_\pi^{S}}$ is one-dimensional.
Thus, the $\LI$-invariant is of codimension at most one.
By the remark at the end of Section \ref{sec-extensions} we have
$c_{\ord_\p}^\epsilon=c_{\infty}^\epsilon \circ \kappa^\ast.$
By Lemma \ref{ordiso} the homomorphism $c_{\infty}^\epsilon$ is an isomorphism, while $\kappa^\ast$ an isomorphism by Proposition \ref{clear}.
\end{proof}

\begin{Rem}
As in \cite{Ge3} one could also define automorphic $\LI$-invariants for higher degree cohomology groups, for which its $\pi$-isotypic component does not vanish.
As these $\LI$-invariants neither seem to show up in exceptional zero formulas nor are they used to define (plectic) Darmon cycles we will not consider them here.
\end{Rem}

\section{P-adic families}
For this section we assume that $F$ is totally real that $\pi_\p(\C)$ is the Steinberg representation and $\pi$ is non-critical at $\p$.
We give a formula for the automorphic $\LI$-invariant in terms of derivatives of $U_\p$-eigenvalues of $p$-adic families passing through $\pi$.
Comparing with the corresponding formula for the Fontaine--Mazur $\LI$-invariant of the corresponding Galois representation we deduce that automorphic and Fontaine--Mazur $\LI$-invariants agree.
\subsection{The weight space}\label{weightspace}
Let $\Omega$ be a finite extension of $E_p$. Define the (partial) weight space $\mathcal{W}_\p$ to be the rigid analytic space over $\Omega$ associated to the completed group algebra $\mathcal{O}_\Omega\llbracket \mathcal{O}_{\p}^\times \rrbracket$.
There is a universal character $$\kappa^\mathrm{un}\colon \mathcal{O}_\p^\times\longrightarrow (\mathcal{O}_\Omega\llbracket \mathcal{O}_{\p}^\times\rrbracket)^\times.$$
Let $\mathcal{U}\subseteq\mathcal{W}_\p$ be an affinoid and $\mathcal{O}(\mathcal{U})$ be the ring of its rigid analytic functions. We will denote by $\kappa^\mathrm{un}_\mathcal{U}\colon \mathcal{O}_\p^\times\rightarrow \mathcal{O}(\mathcal{U})^\times$ the restriction of the universal character to $\mathcal{U}$.
For an affinoid $\mathcal{U}\subseteq \mathcal{W}_\p$ and a locally analytic character $\chi\colon B\cap I_\p\rightarrow \mathcal{O}(\mathcal{U})^\times$ recall from section \ref{overconvergent} the $\mathcal{O}(\mathcal{U})[I_\p]$-module $\mathcal{A}_\chi^n$ defined as the locally $n$-analytic induction of $\chi$ to $I_\p$, and the cohomology groups $\HH^d(X_{K^\p\times I_\p},\mathcal{D}_{\chi,\bm{k}^\p}^n)$. 
If $\chi$ is the universal character $\kappa_{\mathcal{U}}^\mathrm{un}$, we simply write $\HH^d(X_{K^\p\times I_\p},\mathcal{D}_{\mathcal{U},\bm{k}^\p}^n)$ in place of $\HH^d(X_{K^\p\times I_\p},\mathcal{D}_{\kappa_{\mathcal{U}}^\mathrm{un},\bm{k}^\p}^n)$.

\subsection{Etaleness at $\mathfrak{m}_\pi$}
Let $\mathcal{U}$ be an admissible open affinoid in $\mathcal{W}_\p$ containing $\bm{k}_\p$ and let $\mathcal{O}(\mathcal{U})_{\bm{k}_p}$ be the rigid localization of $\mathcal{O}(\mathcal{U})$ at $\bm{k}_\p\in \mathcal{U}$. It is the local ring defined as
$$\mathcal{O}(\mathcal{U})_{\bm{k}_\p}=\varinjlim_{\bm{k}_\p\in \mathcal{U}'\subset \mathcal{U}}\mathcal{O}(\mathcal{U}')$$
where the limit is taken over all admissible open sub-affinoid $\mathcal{U}'_\p$ in $\mathcal{U}$ containing $\bm{k_\p}$. Thus, it contains the algebraic localization of $\mathcal{O}(\mathcal{U})$ at the maximal ideal $\mathfrak{m}_{\bm{k}_\p}$ of $\bm{k}_\p$.

\begin{Thm}\label{etaleness}
Up to shrinking $\mathcal{U}$ to a small enough open affinoid containing $\bm{k}_\p$ the following holds:
$$\HH^d (X_{K^\p\times I_\p},\mathcal{D}_{\mathcal{U},\bm{k}^\p}^n)^\epsilon_{(\mathfrak{m}^S_\pi,U_\p-1)}= 0\quad \mbox{for every}\ d\neq q,$$
for $d=q$ it is a free $\mathcal{O}(\mathcal{U})_{\bm{k}_\p}$-module of finite rank and the map of $\mathcal{O}(\mathcal{U})_{\bm{k}_\p}$-modules obtained by localizing the composition of \eqref{OC} with the map induced by reduction to $\mathcal{D}_{\bm{k}}^n$ induces an isomorphism
$$\HH^q(X_{K^\p\times I_\p},\mathcal{D}_{\mathcal{U},\bm{k}^\p}^n)^\epsilon_{(\mathfrak{m}^S_\pi,U_\p-1)}\otimes_{\mathcal{O}(\mathcal{U})_{\bm{k}_\p}}\mathcal{O}(\mathcal{U})_{\bm{k}_\p}/\mathfrak{m}_{\bm{k}_\p}\xlongrightarrow{\cong} \HH^q(X_{K^\p\times I_\p},V_{\bm{k}}^\vee)^\epsilon_{(\mathfrak{m}^S_\pi,U_\p-1)}.$$
Moreover, the operator $U_\p^\circ$ acts on it via a scalar $\alpha_\p\in \mathcal{O}(\mathcal{U})_{\bm{k}_\p}^\times$.
\end{Thm}

\begin{proof}
Recall that by Theorem \ref{dim} we have
\begin{equation*}
\HH^d(X_{K^\p\times I_\p},V_{\bm{k}}(\Omega)^\vee)_{(\mathfrak{m}_\pi^S,U_\p-1)}=0 \quad \mbox{for all } d\neq q.
\end{equation*}
The first claims follow using the same arguments as in the proof of \cite[Theorem 2.14]{BDJ}.
The statement about the operator $U_\p^\circ$ can be deduced from the fact that
\begin{equation*}
\HH^q(X_{K^\p\times I_\p},V_{\bm{k}}(\Omega)^\vee)^\epsilon_{(\mathfrak{m}_\pi^S,U_\p-1)}
\end{equation*}
is an absolutely irreducible $\mathbb{T}_{K^\p}(\Omega)$-module.
\end{proof}

\subsection{Infinitesimal deformations and $\mathcal{L}$-invariants}

Let $\Omega[\varepsilon]:=\Omega[X]/(X^2)$ be the $\Omega$-algebra of dual numbers over $\Omega$ and $\pi\colon \Omega[\varepsilon]\rightarrow \Omega$ be the the natural surjection sending $\varepsilon$ to $0$. It should be thought as the space of tangent vectors to a fixed $\Omega$-scheme at a fixed point in the following sense. If $X=\text{Spec }(A)$ is an affine $\Omega$-scheme and $x\colon A\rightarrow A/\mathfrak{m}_x=\Omega$ is a $\Omega$-valued point, then the space of morphisms $\bm{v}_x \colon A\rightarrow \Omega[\varepsilon]$ such that $\pi\circ\bm{v}_x=x$ is identified with the tangent space of $X$ at $x$. 

Let $\mathcal{U}$ be an admissible open affinoid containing $\bm{k}$ and $\chi\colon B^\times\rightarrow \mathcal{O}(\mathcal{U})^\times$ a locally analytic character, that we identify with an element of $\Hom(F_\p^\times,\mathcal{O}(\mathcal{U})^\times)$ as in section \ref{ident}.  

Let $\bm{v}\colon \mathcal{O}(\mathcal{U})\rightarrow \Omega[\varepsilon]$ be an element of the tangent space of $\mathcal{U}$ at $\bm{k}$. Then the pullback $\chi_{\bm{v}}=\bm{v}\circ \chi\in \Hom(F_\p^\times, \Omega[\varepsilon])$ of $\chi$ along $\bm{v}$ can be written in a unique way as
\begin{equation}\label{first_order}
\chi_{\bm{v}}=\overline{\chi}(1+\partial_{\bm{v}}(\chi) \varepsilon),
\end{equation}
where $\overline{\chi}\colon F_\p^\times \rightarrow (\mathcal{O}(\mathcal{U})/\mathfrak{m}_{\bm{k}})^\times=\Omega^\times$ denotes the reduction of $\chi$ modulo $\mathfrak{m}_{\bm{k}}$, and $\partial_{\bm{v}}(\chi)$ is a homomorphism $F_\p^\times\rightarrow \Omega$.

Now, assume that $\chi\colon B^\times\rightarrow \mathcal{O}(\mathcal{U})^\times$ is a locally analytic character such that $\chi \pmod{\mathfrak{m}_{\bm{k}}}=\cf_{\bm{k}_\p}$. For an element $\bm{v}$ in the tangent space of $\mathcal{U}$ at $\bm{k}$, write $\chi_{\bm{v}}=\bm{v}\circ \chi$ as above $\chi_{\bm{v}}=\cf_{\bm{k}_\p}(1+\partial_{\bm{v}}(\chi)\varepsilon)$.
Consider the map induced in cohomology by the reduction of $\chi$ modulo $\mathfrak{m}_{\bm{k}}$:
$$\mathrm{red}_{\chi}\colon \HH_{\mathcal{O}(\mathcal{U}),\cont}^q(X_{K^\p}^\p,\mathbb{I}_B(\chi),V_{\bm{k}^\p}^\vee)_{\mathfrak{m}_\pi^S}^\epsilon\longrightarrow \HH^q_{\Omega,\cont}(X_{K^\p}^\p,\mathbb{I}_B(\cf_{\bm{k}_\p}),V_{\bm{k}^\p}^\vee)_{\mathfrak{m}_\pi^S}^\epsilon.$$
Then, the following holds.

\begin{Pro}\label{L-inv}
If $\mathrm{red}_{\chi}$ is surjective, then $\partial_{\bm{v}}(\chi)$ belongs to $\mathcal{L}_\p(\pi)^\epsilon$ for every element $\bm{v}$ of the tangent space of $\mathcal{U}$ at $\bm{k}$.
\end{Pro}

\begin{proof}
The locally analytic character $\chi_{\bm{v}}$ of $B$ over $\Omega[\varepsilon]$ can be seen as a two-dimensional representation $\tau_{\chi_{\bm{v}}}$ of $B$ over $\Omega$. It is in fact the representation that we denoted by $\tau_{\partial_{\bm{v}}(\chi),\bm{k}_\p}$ in section \ref{sec-extensions}. From the discussion in \ref{sec-extensions}, there is a commutative diagram
\begin{equation*}
\begin{tikzcd}
\HH^q(X_{K^\p}^\p,\mathbb{I}_B(\cf_{\bm{k}_\p}),V_{\bm{k}^\p}(\Omega)^\vee)_{\mathfrak{m}_\pi^S}^\epsilon\arrow[r, "\hat{c}_{\partial_{\bm{v}}(\chi)}^\epsilon"] & 
\HH^{q+1}(X_{K^\p}^\p,\mathbb{I}_B(\cf_{\bm{k}_\p}),V_{\bm{k}^\p}(\Omega)^\vee)_{\mathfrak{m}_\pi^S}^\epsilon\arrow[d]\\
\HH^q(X_{K^\p}^\p,\St_{\bm{k}_\p}^{\text{an}}(\Omega),V_{\bm{k}^\p}(\Omega)^\vee)_{\mathfrak{m}_\pi^S}^\epsilon\arrow[r, "c_{\partial_{\bm{v}}(\chi)}^\epsilon"] \arrow[u]& 
\HH^{q+1}(X_{K^\p}^\p,V_{\bm{k}_\p}(\Omega), V_{\bm{k}^\p}(\Omega)^\vee)_{\mathfrak{m}_\pi^S}^\epsilon
\end{tikzcd}
\end{equation*}
where $\hat{c}_{\partial_{\bm{v}}(\chi)}^\epsilon$ and $c_{\partial_{\bm{v}}(\chi)}^\epsilon$ are the boundary maps induced by the dual of the short exact sequences 
\begin{equation}\label{extension}
0\rightarrow \mathbb{I}_B(\cf_{\bm{k}_\p}) \rightarrow\mathbb{I}_B(\tau_{\chi_{\bm{v}}})\rightarrow \mathbb{I}_B (\cf_{\bm{k}_\p}) \rightarrow 0,
\end{equation}
and
\begin{equation*}
0\rightarrow \St_{\bm{k}_\p}^\text{an}(\Omega) \rightarrow\mathcal{E}_{\partial_{\bm{v}}(\chi),\bm{k}_\p}\rightarrow V_{\bm{k}_\p}(\Omega) \rightarrow 0,
\end{equation*}
respectively.
It is sufficient to prove that $\hat{c}_{\partial_{\bm{v}}(\chi)}^\epsilon$ is the zero map. This follows from the fact that the map
$$\HH^q(X_{K^\p}^\p,\mathbb{I}_B(\tau_{\chi_{\bm{v}}}),V_{\bm{k}^\p}(\Omega)^\vee)_{\mathfrak{m}_\pi^S}^\epsilon\rightarrow \HH^q(X_{K^\p}^\p,\mathbb{I}_B(\cf_{\bm{k}_\p}),V_{\bm{k}^\p}(\Omega)^\vee)_{\mathfrak{m}_\pi^S}^\epsilon$$
induced in cohomology by the dual of \eqref{extension} is surjective because it is induced by the reduction $\chi_{\bm{v}}\rightarrow \cf_{\bm{k}_\p}$ of $\chi_{\bm{v}}$ modulo $\varepsilon$, which is the same as the reduction $\chi\rightarrow \cf_{\bm{k}_\p}$ modulo $\mathfrak{m}_{\bm{k}}$ for every $\bm{v}$ by definition of tangent space. 
\end{proof}

Recall that we denoted by $\alpha_\p\in \mathcal{O}(\mathcal{U})^\times_{\bm{k}_\p}$ the eigenvalue of the Hecke operator $U_\p^\circ$ acting on $\HH^q(X_{K^\p\times I_\p},\mathcal{D}^n_{\mathcal{U},\bm{k}_\p})^\epsilon_{(\mathfrak{m}_\pi^S,U_\p-1)}$. Up to shrinking $\mathcal{U}$ we can assume that $\alpha_\p\in\mathcal{O}(\mathcal{U})^\times$. 

Let $\chi_{\alpha_\p}\colon  B^\times\rightarrow \mathcal{O}(\mathcal{U})^\times$ be the character defined by
\begin{align*}
\chi_{\alpha_\p}\vert_{B\cap I_\p} &= \kappa^\mathrm{un}_{\mathcal{U}}\vert_{\mathcal{O}_\p^\times}\\
\chi_{\alpha_\p}(u_\p) &= \alpha_\p.
\end{align*}

Now we are ready to prove the main result of this section.
\begin{Thm}\label{mainthm}
For every element $\bm{v}$ of the tangent space of $\mathcal{U}$ at $\bm{k}$ we have
$$\partial_{\bm{v}}(\chi_{\alpha_\p})\in \mathcal{L}_\p(\pi)^\epsilon.$$
\end{Thm}
\begin{proof}
By Proposition \ref{L-inv}, it is enough to prove that $\mathrm{red}_{\chi}$ is surjective.

By Theorem \ref{Koszul}, with the same arguments as in the proof of Proposition \ref{equivalence}, the map induced in cohomology by \eqref{exact}
$$\HH^q_{\mathcal{O}(\mathcal{U}),\cont}(X_{K^\p}^\p,\mathbb{I}_B(\chi_{\bm{k}_\p}),V_{\bm{k}^\p}^\vee)_{\mathfrak{m}_\pi^S}\too \HH^q(X_{K^\p\times I_\p},\mathcal{D}_{\mathcal{U},\bm{k}^\p}^n)_{(\mathfrak{m}_\pi^S,U_p-1)}$$
is an isomorphism.
Moreover, by Theorem \ref{etaleness}, the reduction mod $\mathfrak{m}_{\bm{k}}$ map induces a surjective map 
$$\HH^q(X_{K^\p\times I_\p},\mathcal{D}_{\mathcal{U},\bm{k}^\p}^n)^\epsilon_{(\mathfrak{m}_\pi^S,U_\p-1)}\longrightarrow \HH^q(X_{K^\p\times I_\p},V_{\bm{k}}(\Omega)^\vee)^\epsilon_{(\mathfrak{m}_\pi^S,U_\p-1)}.$$
On the other hand, one has the isomorphisms
\begin{align*}
\HH^q(X^{\p}_{K^{\p}\times I_\p},V_{\bm{k}}(\Omega)^\vee)^\epsilon_{(\mathfrak{m}_\pi^S,U_\p-1)}
\cong &\HH^q(X_{K^{\p}}^\p,i_B(\cf_{\bm{k}_\p}),V_{\bm{k}^\p}(\Omega)^\vee)^\epsilon_{\mathfrak{m}_\pi^S}\\
\cong &\HH^q_{\Omega,\cont}(X_{K^\p}^\p,\mathbb{I}_B(\cf_{\bm{k}_\p}),V_{\bm{k}^\p}(\Omega)^\vee)^\epsilon_{\mathfrak{m}_\pi^S}\\
\end{align*}
where the first isomorphism follows by the arguments in the proof of Proposition \ref{equivalence} and the second one from the non-criticality of $\pi_\Omega$.

Recollecting all the maps, one gets the claim.
\end{proof}

\subsection{Relation with Galois representations}
Let $\rho=\rho_\pi \colon \Gal(\overline{F}/F) \rightarrow \GL_2(\Omega)$ be the $2$-dimensional Galois representation attached to $\pi$ and let $\rho_\p$ be its restriction to a decomposition group $\Gal(\overline{F_\p}/F_\p)$ at $\p$.
As local-global compatibility is known in this case by Saito (cf.~\cite{Saito}), the representation $\rho_\p$ is semistable, non-crystalline, i.e.:
$$\mathcal{D}_{\mathrm{st}}(\rho_\p)= (\rho_\p\otimes_{\Q_p} B_{\mathrm{st}})^{\Gal(\overline{F_\p}/F_\p)}$$
is a free $\Omega\otimes_{\Q_p} F_{\p,0}$-module (where $F_{\p,0}$ denotes the maximal unramified subfield of $F_\p$) and the nilpotent linear map $N_\p$ inherited from the corresponding map on Fontaine's semistable period ring $B_{\mathrm{st}}$ is non-zero.
Moreover, the kernel of $N_\p$ is a free $\Omega\otimes_{\Q_p} F_{\p,0}$-module of rank one.
Let $e_0$ be a generator of $\Ker(N_\p)$ and define $e_1=N_\p(e_0).$
Furthermore, it is known that the zeroth step of the deRham filtration
$$\mathrm{Fil}^0(\mathcal{D}_{\mathrm{st}}(\rho_\p))\subseteq \mathcal{D}_{\mathrm{st}}(\rho_\p)\otimes_{F_{\p,0}}F_\p$$
is a free $\Omega\otimes_{\Q_p}F_\p$-module of rank one.
In particular there exist $a_0^{\rho_\p},a_1^{\rho_\p}\in \Omega\otimes_{\Q_p}F_\p$ such that
$$\mathrm{Fil}^0(\mathcal{D}_{\mathrm{st}}(\rho_\p))=a_0^{\rho_\p}\cdot e_o + a_1^{\rho_\p} \cdot e_1.$$
\begin{Def}
We call the local Galois representation $\rho_\p$ non-critical if $a_0^{\rho_\p}\in (\Omega\otimes_{\Q_p}F_\p)^\times$.
\end{Def}

It is expected that every $\rho_\p$ coming from a Hilbert modular form as above is non-critical.
If $F_\p=\Q_p$ or $k_\sigma=0$ for all $\sigma\in\Sigma_p$, the fact that $\mathcal{D}_{\mathrm{st}}(\rho_\p)$ is weakly admissible implies that $\rho_\p$ is non-critical.
It seems to the authors of this article that in both, \cite{BDJ}, Section 5.2, and \cite{CMP}, Section 3.2, non-criticality of the Galois representation is assumed implicitly.
Note that the main theorem of \cite{Xie} states that a Galois representation as above is non-critical if $\p$ is the only prime of $F$ above $p$.
But similar as on page 653 of \cite{CMP} an Amice-V\'elu and Vishik-type argument is used in the crucial Proposition 7.3 of \textit{loc.cit.} and it is not clear to the authors of this article, if such an argument is applicable here (see Remark \ref{CMPremark} above).

To any tuple
$$a=(a_\sigma)\in \Omega\otimes_{\Q_p}F_\p \cong \prod_{\sigma\in\Sigma_\p} \Omega$$
we attach a codimension one subspace
$$\LI^{a}\subseteq\Hom(F_\p^\times,\Omega)$$
that does not contain the subspace of smooth homomorphisms as follows:
define
$$\log_\sigma=\log_\p\circ\sigma \colon F_\p^\times\to \Omega,$$
where $\log_\p$ is the usual branch of the $p$-adic logarithm fulfilling $\log_p(p)=0$ and put
$$\LI^{a}=\langle\ord_\p - a_\sigma \log_\sigma \mid \sigma\in\Sigma_\p\rangle,$$
where $\ord_\p$ denotes the normalized $p$-adic valuation of $F_\p^\times$.

\begin{Def}
Suppose $\rho_\p$ is non-critical.
The Fontaine--Mazur $\LI$-invariant of $\rho_\p$ is the codimension one subspace
$$\LI^{FM}(\rho_\p)=\LI^{a_1^{\rho_\p}/a_0^{\rho_\p}}\subseteq \Hom(F_\p^\times,\Omega).$$
\end{Def}

\begin{Thm}\label{Galois}
Suppose that $\pi$ is non-critical at $\p$ and that $\rho_\p$ is non-critical.
Then the equality 
$$\mathcal{L}_\p(\pi)^\epsilon=\mathcal{L}(\rho_\p)^{FM}$$
holds for every sign character $\epsilon$.
In particular, the automorphic $\LI$-invariant $\LI_\p(\pi)^\epsilon$ does not depend on the sign character $\epsilon$.
\end{Thm}
\begin{proof}
This follows directly by comparing Theorem \ref{mainthm} with the corresponding formula on the Galois side (cf.~\cite{ZhangYC}, Theorem 1.1) for the family of Galois representation attached to the family passing through $\pi$.
See \cite{GR}, Theorem 4.1 for more details in case $k_\sigma=0$ for all $\sigma\in\Sigma_\p$.
\end{proof}

In \cite{Ding1} respectively \cite{Ding2} Ding proves that in case $D$ is split at exactly one Archimedean place the Fontaine--Mazur $\LI$-invariant can be detected by completed cohomology of the associated Shimura curve.
Thus, by the theorem above the automorphic $\LI$-invariant can also be detected by completed cohomology in that case.
For the modular curve Breuil gives a direct proof of this consequence in \cite{Br2}.
It would be worthwhile to explore whether Breuil's proof extends to our more general setup.

\bibliographystyle{abbrv}
\bibliography{bibfile}

\end{document}